\documentclass{amsart}
\usepackage{amsmath,amsthm,amssymb}
\usepackage{mathrsfs}
\usepackage{hyperref}
\usepackage{graphicx}
\usepackage{color}

\usepackage[all]{xy}
\SelectTips{cm}{10}

\usepackage[small,nohug,heads=LaTeX,midshaft]{diagrams}
\diagramstyle[labelstyle=\scriptstyle]
\newarrow{Inline}C----
\newarrow{Corresponds}<--->

\numberwithin{figure}{section}
\numberwithin{equation}{section}

\theoremstyle{plain}
\newtheorem{thm}{Theorem}[section]
\newtheorem{prop}[thm]{Proposition}
\newtheorem{lem}[thm]{Lemma}
\newtheorem{cor}[thm]{Corollary}

\theoremstyle{definition}
\newtheorem{defn}[thm]{Definition}
\newtheorem{example}[thm]{Example}
\newtheorem{problem}[thm]{Problem}

\theoremstyle{remark}
\newtheorem{rem}[thm]{Remark}

\newcommand{\la}{\langle}
\newcommand{\ra}{\rangle}
\newcommand{\into}{\hookrightarrow}
\newcommand{\onto}{\twoheadrightarrow}

\newcommand{\xra}{\xrightarrow}

\newcommand{\wt}{\widetilde}

\newcommand{\om}{\Omega}
\newcommand{\si}{\Sigma}

\newcommand{\lsi}{\Omega\Sigma}
\newcommand{\lsii}{\Omega^2\Sigma^2}
\newcommand{\lsn}{\Omega^n\Sigma^n}
\newcommand{\lsoo}{\Omega^{\infty} \Sigma^{\infty}}

\newcommand{\ms}{\mathscr}

\newcommand{\Z}{\mathbb{Z}}

\newcommand{\id}{\mathrm{id}}
\newcommand{\im}{\mathrm{Im}\,}
\newcommand{\Ker}{\mathrm{Ker}\,}

\newcommand{\Aut}{\mathrm{Aut}}

\begin{document}

\title{Group Operads and Homotopy Theory}
\author{Wenbin ZHANG}
\date{19 Jun 2012}
\address{Department of Mathematics, National University of Singapore}
\email{smile.wenbin@gmail.com}

\begin{abstract}
  We introduce the classical theory of the interplay between group theory and topology into the context of operads and explore some applications to homotopy theory. We first propose a notion of a group operad and then develop a theory of group operads, extending the classical theories of groups, spaces with actions of groups, covering spaces and classifying spaces of groups. In particular, the fundamental groups of a topological operad is naturally a group operad and its higher homotopy groups are naturally operads with actions of its fundamental groups operad, and a topological $K(\pi,1)$ operad is characterized by and can be reconstructed from its fundamental groups operad. Two most important examples of group operads are the symmetric groups operad and the braid groups operad which provide group models for $\Omega^{\infty} \Sigma^{\infty} X$ (due to Barratt and Eccles) and $\Omega^2 \Sigma^2 X$ (due to Fiedorowicz) respectively. We combine the two models together to produce a free group model for the canonical stabilization $\Omega^2 \Sigma^2 X \hookrightarrow \Omega^{\infty} \Sigma^{\infty} X$, in particular a free group model for its homotopy fibre.
\end{abstract}

\maketitle

\tableofcontents

\section{Introduction}
The objective of this paper is to introduce the classical theory of the interplay between group theory and topology into the context of operads and explore applications to homotopy theory. In particular it serves as a tool for the establishment of certain operations on $\ms{C}$-spaces in \cite{Zhang:arXiv2011:OSOAHG}.

In P. May's definition of a symmetric operad \cite{May:1972:GILS}, symmetric groups $S_n$ ($n\geq 0$) play a special role. In addition, Barratt and Eccles used all $S_n$ to construct a model for $\lsoo X$ around 1970 \cite{Bar-Ecc:1974:I, Bar-Ecc:1974:II, Bar-Ecc:1974:III}. In early 1990's, Fiedorowicz \cite{Fiedorowicz:preprint:SBC} observed that symmetric groups can be replaced by braid groups $B_n$, $n\geq 0$, to define braided operads, and used all $B_n$ to construct a model for $\lsii X$. Later on Tillmann (2000) \cite{Tillmann:2000:HGSODILS} proposed an idea of constructing operads from families of groups and her student, Wahl, then gave a more detailed study of this construction and used ribbon braid groups to construct a model for $\lsii (S^1_+ \wedge X)$ in her Ph.D. thesis (2001) \cite{Wahl:2001:PhD}. Observing that all these examples can be treated in a uniform way, it is then natural to ask:
\begin{quote}
  \textbf{Question:} Can these canonical examples lead to a general theory? If yes, is this general theory natural and interesting?
\end{quote}
Motivated by the above mentioned work and an investigation of the fundamental groups of topological operads, a notion of group operads is proposed and a general theory is developed in this paper.

A group operad $\ms{G}= \{G_n\}_{n\geq 0}$ is an operad with a morphism to the symmetric groups operad $\ms{S}= \{S_n\}_{n\geq 0}$ such that 1) each $G_n$ is a group and $G_n\to S_n$ is a homomorphism, 2) the identity of $G_1$ is the unit of the operad $\ms{G}$ and 3) the composition $\gamma$ is a crossed homomorphism, namely
$$\gamma (aa'; b_1b'_1, \ldots, b_kb'_k)= \gamma (a;b_1, \ldots, b_k) \gamma (a'; b'_{a^{-1}(1)}, \ldots, b'_{a^{-1}(k)}).$$
Canonical examples of group operads are the symmetric groups operad $\ms{S}$, the braid groups operad $\ms{B}$ and the ribbon braid groups operad $\ms{R}$.

Group operads play a role like groups. As actions of groups on spaces, actions of a group operad $\ms{G}$ on other operads can be defined and an operad $\ms{C}$ with an action of $\ms{G}$ is called a $\ms{G}$-operad. As such a nonsymmetric operad is an operad with an action of the trivial group operad and a symmetric operad is an operad with an action of the symmetric groups operad. The theory of symmetric operads can then be generalized to $\ms{G}$-operads.

Besides the above canonical examples, a construction has been found to extend any group to a group operad and any $G$-space to a $\ms{G}$-operad, and thus provides countless examples of group operads and $\ms{G}$-operads, cf. Remarks \ref{rem:extending_group_group-operad} and \ref{rem:extending_G-space_G-operad}. This construction will appear elsewhere.

The idea of a group operad is mainly motivated by a few canonical examples, however it turns out that group operads are actually natural by investigating operad structures on the homotopy groups of topological operads. We find that the operad structure of a topological operad naturally induces operad structures on its homotopy groups such that its fundamental groups is a group operad and its higher homotopy groups are operads with actions of its fundamental groups operad.

The classical theory of covering spaces is extended to a theory of covering operads, by which we establish relationships between group operads and topological $K(\pi,1)$ operads analogous to the one between groups and $K(\pi,1)$ spaces. The usual construction of the classifying space of a group can be used to construct a topological $K(\pi,1)$ operad for a group operad $\ms{G}$, thought of as the classifying operad of $\ms{G}$, with $\ms{G}$ realized as its fundamental groups operad. In addition, a nice topological $K(\pi,1)$ operad is characterized by and can be reconstructed from its fundamental groups operad.

Group operads can apply to homotopy theory via the associated monads of their classifying operads and in particular may be used to produce algebraic models for certain canonical objects in homotopy theory. For instance as mentioned at the beginning, the symmetric groups operad and the braid groups operad give algebraic models for $\lsoo X$ (Barratt-Eccles \cite{Bar-Ecc:1974:I}) and $\lsii X$ (Fiedorowicz \cite{Fiedorowicz:preprint:SBC}) respectively. We combine the two models together to produce a free group model for the canonical stabilization $\lsii X\into \lsoo X$, in particular a free group model for the homotopy fibre of this stabilization. A few canonical filtrations of $\lsii X$ can also be constructed. Further applications of these models will be investigated in future.

\medskip
\textbf{Notations and conventions.}
For $\sigma,\tau\in S_n$ where $S_n$ is the symmetric group of $\{1,\ldots,n\}$, their product is defined as
$$\sigma \cdot \tau := \tau \circ \sigma: \{1,\ldots,n\} \xra{\sigma} \{1,\ldots,n\} \xra{\tau} \{1,\ldots,n\}, \quad (\sigma \cdot \tau)(i)= \tau (\sigma(i)).$$
Let $S_k$ act on symmetric operads from the left and on $X^k$ from the right.

Whenever we have a group $G$ with a homomorphism $\pi: G\to S_n$, we shall regard $g(i)= (\pi g)(i)$ for $g\in G$ and $1\leq i\leq n$.

Two different label systems of $\Delta$-sets (simplicial sets) are used here. One is the usual one starting from 0, $X=\{X_n\}_{n\geq 0}$ with $d_i: X_{n+1}\to X_n$ for $0\leq i\leq n+1$ (and $s_i: X_n\to X_{n+1}$ for $0\leq i\leq n$), for $n\geq 0$; and another one shifts 0 to 1, i.e., starting from 1, $X=\{X_n\}_{n\geq 1}$ with $d_i: X_{n+1}\to X_n$ for $1\leq i\leq n+1$ (and $s_i: X_n\to X_{n+1}$ for $1\leq i\leq n$), for $n\geq 1$. The latter is used for operads, like the symmetric groups operad, braid groups operad, etc.

For any symbol $a$, let $a^{(k)}$ denote the $k$-tuple $(a,\ldots,a)$. For any $n$-tuple $(a_1, \ldots, a_n)$, let $(a_1, \ldots, \hat{a}_i, \ldots, a_n)= (a_1, \ldots, a_{i-1}, a_{i+1}, \ldots, a_n)$, i.e., $a_i$ is omitted.

Throughout this paper, all topological spaces are assumed to be compactly generated Hausdorff spaces \cite{Steenrod:1967:CCTS}. For two points $x, x'$ in a space $X$, let $x\sim x'$ denote that $x,x'$ are in the same path-connected component of $X$.

\section{Group Operads and Operads with Actions of Group Operads}
In this section, we define a notion of group operads, study their basic properties, and then discuss topological and simplicial operads with actions of group operads.

\subsection{Operads and Some Basic Aspects}
We recall P. May's definition of operads and symmetric operads, and discuss some basic aspects of operads concerning basepoints and simplicial structure.

\begin{defn}
  A \textbf{topological operad} $\ms{C}$ consists of
  \begin{itemize}
    \item[1)] a sequence of topological spaces $\{\ms{C}(n)\}_{n\geq 0}$ with $\ms{C}(0)=*$,
    \item[2)] a family of maps,
        $$\gamma: \ms{C}(k)\times \ms{C}(m_1)\times \cdots \times \ms{C}(m_k)\to \ms{C}(m), \quad k\geq 1, m_i\geq 0, m=m_1+ \cdots+ m_k,$$
    \item[3)] an element $1\in \ms{C}(1)$ called the \emph{unit},
  \end{itemize}
  satisfying the following two coherence properties: for $a\in \ms{C}(k)$, $b_i\in \ms{C}(m_i)$, and $c_j\in \ms{C}(n_j)$, $n_j\geq 0$,
  \begin{itemize}
    \item[i)] \emph{Associativity}:
        $$\gamma( \gamma( a; b_1, \ldots, b_k); c_1, \ldots, c_m)= \gamma( a; \gamma( b_1; c_1, \ldots, c_{m_1}), \ldots, \gamma( b_k; c_{m_1+ \cdots+ m_{k-1}+1}, \ldots, c_m));$$
    \item[ii)] \emph{Unitality}: $\gamma (1;a)= a$ and $\gamma (a; 1^{(k)})=a$.
  \end{itemize}
  A \textbf{symmetric topological operad} is a topological operad with a left action of $S_n$ on $\ms{C}(n)$ for each $n$, satisfying the following \emph{equivariance} property: for $\sigma \in S_k$, and $\tau_i\in S_{m_i}$,
  $$\gamma (\sigma a; \tau_1 b_1, \ldots, \tau_k b_k)= \gamma (\sigma; \tau_1, \ldots, \tau_k) \gamma (a; b_{\sigma^{-1}(1)}, \ldots, b_{\sigma^{-1}(k)}).$$
\end{defn}

The most important examples of symmetric topological operads are perhaps the little $n$-cubes operads, refer to \cite{May:1972:GILS} for details. A very important symmetric discrete operad is the symmetric groups operad $\ms{S}= \{S_n\}_{n\geq 0}$ introduced in Definition 3.1 (i) of \cite{May:1972:GILS}.

Basepoints of operads are discussed next.

\begin{defn}
  A basepoint of a topological operad $\ms{C}$ is a sequence of points $\{e_k\}_{k\geq 0}$ with $e_1=1\in \ms{C}(1)$, $e_k\in \ms{C}(k)$ such that $\gamma (e_k; e_{m_1}, \ldots, e_{m_k}) \sim e_m$. A strict basepoint is a basepoint such that $\gamma (e_k; e_{m_1}, \ldots, e_{m_k})= e_m$. A symmetric (strict) basepoint of a symmetric operad is a (strict) basepoint such that $e_k\sim \sigma e_k$ for all $\sigma\in S_k$.
\end{defn}

For a discrete operad, a basepoint is obviously strict. The little $n$-cubes operads $\ms{C}_n$, however, does not have any strict basepoint, but does have a basepoint.

\begin{prop}
  $\ms{C}$ has a basepoint iff there exists $a\in \ms{C}(2)$ such that $d_1a\sim d_2a\sim 1\in \ms{C}(1)$ and $\gamma (a;1,a)\sim \gamma (a;a,1)$.
\end{prop}

Thus a basepoint of $\ms{C}$ is determined by such an $a\in \ms{C}(2)$.

\begin{proof}
  Suppose there exists $a\in \ms{C}(2)$ such that $d_1a\sim d_2a\sim 1\in \ms{C}(1)$ and $\gamma (a;1,a)\sim \gamma (a;a,1)$. Let $a_0=*$, $a_1=1$, $a_2=a$, $a_k= \gamma (a; a_{k-1},1)$ for $k\geq 2$. Next check that $\{a_k\}_{k\geq 0}$ is a basepoint of $\ms{C}$.

  First prove $\gamma (a_2; a_i,a_j)\sim a_{i+j}$ by induction on $i+j$. This is evident if $i+j\leq 2$. Assume $\gamma (a_2; a_i,a_j)\sim a_{i+j}$ if $i+j\leq m$. Now suppose $i+j= m+1$. If $j=0$,
  $$\gamma (a_2; a_i,*)= \gamma (d_2a_2; a_i) \sim \gamma (1;a_i)=a_i;$$
  if $j=1$, $\gamma (a_2; a_i, 1)= a_{i+1}$ by definition; if $j>1$,
  \begin{align*}
    \gamma (a_2; a_i,a_j) & = \gamma (a_2; a_i, \gamma (a_2; a_{j-1},1))= \gamma (\gamma (a_2;1,a_2); a_i, a_{j-1},1) \\
    & \sim \gamma (\gamma (a_2;a_2,1); a_i, a_{j-1},1)= \gamma (a_2; \gamma (a_2; a_i, a_{j-1}), 1) \\
    & \sim \gamma (a_2; a_{i+j-1}, 1)= a_{i+j}.
  \end{align*}

  Next prove $\gamma (a_k; a_{m_1}, \ldots, a_{m_k})\sim a_{m_1+ \cdots+ m_k}$ by induction on $k$. Assume this is true for $k\geq 2$. Then for $k+1$,
  \begin{align*}
    \gamma (a_{k+1}; a_{m_1}, \ldots, a_{m_{k+1}}) & = \gamma (\gamma (a_2; a_k,1); a_{m_1}, \ldots, a_{m_{k+1}}) \\
    & = \gamma (a_2; \gamma (a_k; a_{m_1}, \ldots, a_{m_k}), a_{m_{k+1}}) \\
    & \sim \gamma (a_2; a_{m_1+ \cdots+ m_k}, a_{m_{k+1}}) \\
    & \sim a_{m_1+ \cdots+ m_{k+1}}.
  \end{align*}
  So $\{a_k\}_{k\geq 0}$ is a basepoint by definition.
\end{proof}

Note that $a\sim \tau a\in \ms{C}(2)$ can not imply $\gamma (a;a,1)\sim \gamma (a;1,a)$. Thus path-connectivity of $\ms{C}(2)$ is not sufficient for the existence of a basepoint.

For a topological operad $\ms{C}$, define
$$d_i: \ms{C}(k+1)\to \ms{C}(k), \quad d_ia= \gamma (a;1^{i-1},*,1^{k-i})$$
for $1\leq i\leq k+1$. If $\ms{C}$ has a basepoint $\{e_k\}_{k\geq 0}$, define
$$s_i: \ms{C}(k)\to \ms{C}(k+1), \quad s_ia= \gamma (a;1^{i-1},e_2,1^{k-i})$$
for $1\leq i\leq k$. By definition,
$$d_ie_{k+1}= \gamma (e_{k+1}; 1^{i-1}, *, 1^{k-i})\sim e_k, \quad s_ie_k= \gamma (e_k; 1^{i-1}, e_2, 1^{k-i})\sim e_{k+1}$$
for all $i$ and $k$.

\begin{prop}
  Let $\ms{C}$ be a topological operad. Then $\ms{C}$ is a $\Delta$-space. If $\ms{C}$ has a strict basepoint, then $\ms{C}$ is a simplicial space. If $\ms{C}$ has a basepoint, then $\ms{C}$ is a simplicial space up to homotopy, i.e., those simplicial identities hold up to homotopy. \qed
\end{prop}

For a simplicial set, we always have $d_is_i=\id$, thus each $s_i$ is injective. Hence

\begin{cor}
   If $\ms{C}$ is a topological operad with a strict basepoint, then each $s_i: \ms{C}(k)\to \ms{C}(k+1)$ is injective for $k\geq 1$, $1\leq i\leq k$. \qed
\end{cor}

\subsection{Group Operads}
\begin{defn}
  A \textbf{group operad} is an operad $\ms{G}=\{G_n\}_{n\geq 0}$ together with a morphism $\pi: \ms{G}\to \ms{S}$ of operads such that
\begin{itemize}
  \item[1)] each $G_n$ is a group and $\pi: G_n\to S_n$ is a homomorphism,
  \item[2)] the identity $e_1$ of $G_1$ is the unit of the operad $\ms{G}$, and
  \item[3)] $\gamma$ is a crossed homomorphism, i.e., for $a'\in G_k$, $b_i'\in G_{m_i}$,
      $$\gamma (aa'; b_1b'_1, \ldots, b_kb'_k) = \gamma (a;b_1,\ldots,b_k) \gamma(a'; b'_{a^{-1}(1)}, \ldots, b'_{a^{-1}(k)}).$$
\end{itemize}
$\ms{G}$ is called \textbf{non-crossed} if all $\pi$ are trivial (thus $\gamma$ is a homomorphism), and \textbf{crossed} otherwise. A \textbf{morphism} $\psi: \ms{G}\to \ms{G}'$ of group operads is a sequence of homomorphisms $\psi_n: G_n\to G'_n$ such that $\pi' \circ \psi= \pi$ and $\psi$ commutes with $\gamma$, namely
$$\psi_m (\gamma (a; b_1, \ldots, b_k))= \gamma' (\psi_k(a); \psi_{m_1}(b_1), \ldots, \psi_{m_k}(b_k)).$$
A \textbf{sub group operad} $\ms{H}= \{H_n\}_{n\geq 0}$ of a group operad $\ms{G}= \{G_n\}_{n\geq 0}$, written $\ms{H}\leq \ms{G}$, is a sequence of subgroups $H_n\leq G_n$ closed under $\gamma$, together with the restrictions $\pi= \pi|_{\ms{H}}$ and $\gamma= \gamma|_{\ms{H}}$. A sub group operad $\ms{H}$ of $\ms{G}$ is called \textbf{normal}, written $\ms{H} \unlhd \ms{G}$, if $H_n \unlhd G_n$ for each $n$.
\end{defn}

The symmetric groups operad $\ms{S}$ is clearly a crossed group operad and $\pi: \ms{G}\to \ms{S}$ is a morphism of group operads for any group operad $\ms{G}$. It should be noted that a crossed group operad can NOT be regarded as a non-crossed group operad due to the ``crossed homomorphism'' property, and that there is no morphism from a crossed group operad to a non-crossed group operad, also due to the ``crossed homomorphism'' property and additionally to that $\pi$ is nontrivial but $\pi' \circ \psi$ is trivial. Later we shall see that for a morphism $\pi: \ms{G}\to \ms{S}$ of group operads, $\pi: G_n\to S_n$ can only be either trivial for all $n$ or surjective for all $n$.

The notion of a group operad is equipped with these data and properties because a few canonical examples and the fundamental groups of topological operads have them. One may also consider analogous notions in other appropriate categories. The term ``crossed homomorphism'' is chosen due to its relation with crossed simplicial groups.

For any group operad $\ms{G}$, let $e_n$ denote the identity element of $G_n$. If $\ms{G}$ is crossed, then $\gamma$ is not a homomorphism. However, we have $\gamma(e_k; e_{m_1}, \ldots, e_{m_k})= e_m$, since
$$\gamma(e_k; e_{m_1}, \ldots, e_{m_k})= \gamma(e_k^2; e_{m_1}^2, \ldots, e_{m_k}^2)= \gamma(e_k; e_{m_1}, \ldots, e_{m_k})^2.$$
For $a\in G_k$, $b_i\in G_{m_i}$,
\begin{align*}
  \gamma (a, b_1, \ldots, b_k) & = \gamma (e_k, b_1, \ldots, b_k) \gamma (a, e_{m_1}, \ldots, e_{m_k}) \\
  & = \gamma (a, e_{m_1}, \ldots, e_{m_k}) \gamma (e_k; b_{a^{-1}(1)}, \ldots, b_{a^{-1}(k)});
\end{align*}
$$e_m= \gamma(aa^{-1}; b_1 b_1^{-1}, \ldots, b_k b_k^{-1})= \gamma (a, b_1, \ldots, b_k) \gamma (a^{-1}, b_{a^{-1}(1)}^{-1}, \ldots, b_{a^{-1}(k)}^{-1});$$
hence,
$$\gamma (a,b_1,\ldots, b_k)^{-1}= \gamma (a^{-1}, b_{a^{-1}(1)}^{-1}, \ldots, b_{a^{-1}(k)}^{-1}).$$

\begin{prop}
  $\gamma (a;e_k)$ is in the center of $G_k$ for any $a\in G_1$ and $k\geq 1$.
\end{prop}

\begin{proof}
  For $a\in G_1$, $b\in G_k$, $\gamma (a;b)= \gamma (ae_1; e_kb)= \gamma (a;e_k) \gamma (e_1;b)= \gamma (a;e_k)b$, $\gamma (a;b)= \gamma (e_1a; be_k)= \gamma (e_1;b) \gamma (a;e_k)= b\gamma (a;e_k)$, thus $\gamma (a;e_k)b= b\gamma (a;e_k)$, i.e., $\gamma (a;e_k)$ is in the center of $G_k$.
\end{proof}

\begin{cor}
  $G_1$ is abelian.
\end{cor}

\begin{proof}
  For any $a\in G_1$, $a= \gamma (a;e_1)$ is in the center of $G_1$, thus $G_1$ is abelian.
\end{proof}

\begin{example}\label{example:trivial_group_operad}
  The simplest group operad is the trivial group operad $\ms{J}$ with each group the trivial group. This notation $\ms{J}$ is chosen due to the relation between the trivial group operad and James' construction \cite{James:1955:RPS} which is usually denoted $JX$.
\end{example}

\begin{example}
  The \textbf{braid groups operad} is $\ms{B}= \{B_n\}_{n\geq 0}$ together with
  $$\gamma: B_k \times B_{m_1} \times \cdots \times B_{m_k} \to B_m, \quad k\geq 1, m_i\geq 0,$$
  sending $(a;b_1,\ldots,b_k)$ where $a\in B_k$ and $b_i\in B_{m_i}$, to a braid in $B_m$ obtained by replacing the $i$th strand of $a$ by the braid $b_i$. If $m_i=0$, then the $i$th strand of $c$ is deleted. $\ms{B}$ is a crossed group operad.
\end{example}

Since $\gamma (a;b_1,\ldots,b_k)\in P_m$ if $a\in P_k$, $b_i\in P_{m_i}$, by restricting $\gamma$ to the pure braid groups, we have a normal sub group operad of $\ms{B}$.

\begin{example}
  The \textbf{pure braid groups operad} is $\ms{P}= \{P_n\}_{n\geq 0}$ where $P_0$ is trivial, together with the restriction of $\gamma$ of the braid groups operad to pure braid groups. $\ms{P}$ is non-crossed.
\end{example}

\begin{example}
  The \textbf{ribbon braid groups operad} is $\ms{R}= \{R_n\}_{n\geq 0}$ with $R_n= B_n \ltimes \Z^n$ \cite{Tillmann:2000:HGSODILS, Wahl:2001:PhD}. A ribbon braid is a braid with each strand a ribbon. $R_n$ is the braid group on $n$ ribbons with each ribbon twisted a multiple of the full twist. The $\gamma$ of $\ms{R}$ is almost the same as the one of $\ms{B}$. We just need to emphasize that for $(k,\alpha)\in R_1\times R_n= \Z\times R_n$, $\gamma (k;\alpha)\in R_n$ is the ribbon braid obtained by fully twisting $\alpha$ $k$ times.
\end{example}

\begin{example}
  For any abelian group $G$, $\{G^n\}_{n\geq 0}$ is a non-crossed group operad with $\gamma: G^k \times G^{m_1} \times \cdots \times G^{m_k}\to G^m$, $\gamma((a_1,\ldots,a_k); (b_{11}, \ldots, b_{1m_1}), \ldots)= (a_1+ b_{11}, \ldots, a_1+ b_{1m_1}, \ldots, a_i+ b_{ij}, \ldots, a_k+ b_{km})$.
\end{example}

\begin{rem}\label{rem:extending_group_group-operad}
  For any group $G$ with a homomorphism $\pi: G\to S_2$, a construction has been found to extend $G$ to a group operad $\ms{F}(G,\pi)$ with $\ms{F}(G,\pi)(1)=1$ and $\ms{F}(G,\pi)(2)=G$, such that $\ms{F}(G,\pi)$ is non-crossed if $\pi$ is trivial and crossed if $\pi$ nontrivial. For example, for $S_2$ and the identity $\id_{S_2}: S_2\to S_2$, $\ms{F}(S_2, \id_{S^2})$ is isomorphic to the symmetric groups operad $\ms{S}$. Thus the construction $\ms{F}(G,\pi)$ seems interesting and provides countless examples of group operads. Details of this construction will appear elsewhere.
\end{rem}

\subsection{Sub Group Operads and Quotients}
For any group operad $\ms{G}= \{G_n\}_{n\geq 0}$, the trivial group operad $\ms{J}= \{e_n\}_{n\geq 0}$ and $\ms{G}$ are clearly sub group operads. A sub group operad $\ms{H}\leq \ms{G}$ is called \textbf{proper} if it is not $\ms{J}$ nor $\ms{G}$. For a morphism of group operads $\psi: \ms{G}\to \ms{G}'$, $\Ker \psi$ is a normal sub group operad of $\ms{G}$, and $\im \psi$ is a sub group operad of $\ms{G}'$. Clearly $\Ker \psi$ is non-crossed and $\im \psi$ has the same type than $\ms{G}$, i.e., it is non-crossed if $\ms{G}$ is non-crossed, and crossed if $\ms{G}$ is crossed (then $\ms{G}'$ has to be crossed too).

The sub group operad $\ms{H}$ of $\ms{G}$ generated by $A=\{A_n\}_{n\geq 0}$ with each $A_n$ a subset of $G_n$ is the intersection of all sub group operads of $\ms{G}$ containing $A$.

\begin{lem}
  The symmetric groups operad $\ms{S}$ is generated by $(1,2)\in S_2$ since for the generators of $S_n$ ($n\geq 3$), $(1,2)= \gamma (e_2; (1,2), e_{n-2})$, $(n-1,n)= \gamma (e_2; e_{n-2}, (1,2))$ and $(i,i+1)= \gamma (e_3; e_{i-1}, (1,2), e_{n-i-1}) $ for $2\leq i\leq n-2$. \qed
\end{lem}

\begin{prop}
  The symmetric groups operad $\ms{S}$ has no proper sub group operad; namely $\ms{S}$ has only two sub group operads, the trivial one $\ms{J}$ and $\ms{S}$ itself.
\end{prop}

\begin{proof}
  Let $\ms{G}$ be a sub group operad of $\ms{S}$ with a nonidentity permutation $\sigma\in G_n \leq S_n$, then there is $i$ such that $\sigma(i)>i$. Note that
  $$\gamma (\sigma; *^{(i-1)}, e_1, *^{(\sigma(i)-i-1)}, e_1, *^{(n- \sigma(i))})= (1,2)\in S_2,$$
  where $*\in S_0$, thus $(1,2)\in G_2$. Hence $\ms{G}= \ms{S}$ by the lemma.
\end{proof}

\begin{cor}
  If $\ms{G}$ is crossed, then all $\pi: G_n\to S_n$ are epimorphisms. \qed
\end{cor}

Unlike the symmetric groups operad, the braid groups operad $\ms{B}$ has many proper sub group operads. The pure braid groups operad $\ms{P}$ is a canonical one. The ribbon braid groups operad has even more sub group operads. In the following we construct a sequence of sub group operads of $\ms{B}$. Let $\ms{B}^{(k)}= \{B_n^{(k)}\}_{n\geq 0}$ be the sub group operad of $\ms{B}$ generated by $\sigma_i^k$, $i\geq 1$, where $\sigma_1, \ldots, \sigma_{n-1}$ are the standard generators of $B_n$ and the two $\sigma_i$'s of $B_n$ and $B_m$ ($1\leq i<n,m$) are regarded as the same. Obviously $\ms{B}^{(1)}= \ms{B}$, $\ms{B}^{(2)}< \ms{P}$.

First we describe how $\ms{B}^{(k)}$ is generated. Let $A_{n,1}^k= \{\sigma_1^k, \ldots, \sigma_{n-1}^k\}$ and $A_{n,1}'^k$ the subgroup of $B_n$ generated by $A_{n,1}^k$. If $A^k_i= \{A_{n,i}^k\}_{k\geq 0}$ and $A'^k_i= \{A_{n,i}^k\}_{k\geq 0}$ have been defined for $i<m$, let $A^k_m= \{A_{n,m}^k\}_{k\geq 0}$ be the sequence of subsets of $\ms{B}$ generated by $A'^k_{m-1}= \{A_{n,m-1}^k\}_{k\geq 0}$ under $\gamma$, and $A'^k_m= \{A_{n,m}'^k\}_{k\geq 0}$ the sequence of subgroups of $\ms{B}$ generated by $A_{n,m}^k$. Clearly $A'^k_i\subseteq A'^k_{i+1}$ and $\ms{B}^{(k)}= \bigcup_{i\geq 0} A'^k_i$.

Let $\phi: B_n\onto B_n/[B_n,B_n] \cong \Z$ be the canonical homomorphism.

\begin{lem}
  For $n\geq 2$, $m\geq 1$, if $a\in A_{n,m}^k$ or $a\in A_{n,m}'^k$, then $k| \phi \gamma (a; e_{m_1}, \ldots, e_{m_n})$ for any $m_i\geq 0$.
\end{lem}

\begin{proof}
  Note that $\phi \gamma (\sigma_1, e_i, e_j)= ij$, $\phi \gamma (\sigma_1^{-1}, e_i, e_j)= -ij$, $\phi \gamma (\sigma_i, e_{m_1}, \ldots, e_{m_n})= m_i m_{i+1}$, $\phi \gamma (\sigma_i^{-1}, e_{m_1}, \ldots, e_{m_n})= -m_i m_{i+1}$, and $\phi \gamma (e_n, b_1, \ldots, b_n)= \phi b_1+ \cdots+ \phi b_n$. Since $\gamma$ is a crossed homomorphism, $$\phi \gamma (\sigma_1^k, e_i, e_j)= \phi \gamma (\sigma_1, e_i, e_j)+ \phi \gamma (\sigma_1^{k-1}, e_j, e_i)= \cdots= kij;$$
  similarly $\phi \gamma (\sigma_i^k, e_{m_1}, \ldots, e_{m_n})= km_i m_{i+1}$ and
  $$\phi \gamma (\sigma_i^k \sigma_j^k, e_{m_1}, \ldots, e_{m_k})= km_i m_{i+1}+ km_{\sigma_i^{-k}(j)} m_{\sigma_i^{-k}(j+1)}.$$
  So the assertion holds for $m=1$. Suppose it holds for $i<m$. For $a\in A_{n,m}^k$, $a= \gamma (a'; a'_1, \ldots, a'_{n'})$ for some $a'\in A_{n',m-1}'^k$, $a'_i\in A_{?,m-1}'^k$. Then
  \begin{align*}
    &\ \phi \gamma (a; e_{m_1}, \ldots, e_{m_n}) \\
    = &\ \phi \gamma (\gamma (a'; a'_1, \ldots, a'_{k'}); e_{m_1}, \ldots, e_{m_n}) \\
    = &\ \phi \gamma (a'; \gamma (a'_1; e_{m_1}, \ldots), \ldots, \gamma (a'_{k'}; \ldots, e_{m_n})) \\
    = &\ \phi \gamma (e_{n'}; \gamma (a'_1; e_{m_1}, \ldots), \ldots, \gamma (a'_{k'}; \ldots, e_{m_n}))+ \phi \gamma (a'; e_?, \ldots, e_?) \\
    = &\ \phi \gamma (a'_1; e_{m_1}, \ldots)+ \cdots+ \phi \gamma (a'_{k'}; \ldots, e_{m_n})+ \phi \gamma (a'; e_?, \ldots, e_?).
  \end{align*}
  (Some subscripts of $e$ are omitted for convenience but it should be clear what they are.) Thus by the induction assumption, $k| \phi \gamma (a; e_{m_1}, \ldots, e_{m_n})$. For $a,b\in A_{n,m}^k$,
  $$\phi \gamma (ab; e_{m_1}, \ldots, e_{m_n})= \phi \gamma (a; e_{m_1}, \ldots, e_{m_n})+ \phi \gamma (b; e_{m_{a^{-1}(1)}}, \ldots, e_{m_{a^{-1}(n)}}),$$
  thus the assertion holds for $c\in A_{n,m}'^k$ as well.
\end{proof}

\begin{lem}
  For $n\geq 2$ and $m\geq 1$, $\phi A_{n,m}^k= k\Z$, $\phi A_{n,m}'^k= k\Z$.
\end{lem}

\begin{proof}
  It is obvious that $\phi A_{n,1}^k= k\Z$, $\phi A_{n,1}'^k= k\Z$, $\phi A_{n,m}^k \supseteq k\Z$, $\phi A_{n,m}'^k \supseteq k\Z$ for $m\geq 1$. Suppose $\phi A_{n,i}^k= k\Z$, $\phi A_{n,i}'^k= k\Z$ for $i<m$.
  \begin{align*}
    \phi \gamma (a; b_1, \ldots, b_n) & = \phi \gamma (e_n; b_1, \ldots, b_n)+ \phi \gamma (a; e_?, \ldots, e_?) \\
    & = \phi b_1+ \cdots+ \phi b_n+ \phi \gamma (a; e_?, \ldots, e_?).
  \end{align*}
  So the assertion holds by the previous lemma.
\end{proof}

\begin{prop}
  $\ms{B}^{(k)}$ is a proper subgroup operad of $\ms{B}^{(k')}$ if $k'|k$. $\ms{B}^{(2k)}$ is non-crossed while $\ms{B}^{(2k+1)}$ is crossed.
\end{prop}

\begin{proof}
  It is clear that $\ms{B}^{(k)}\leq \ms{B}^{(k')}$ if $k'|k$. The proposition holds by the previous lemma.
\end{proof}

We next discuss quotient group operads. For a subgroup $H$ of a group $G$, let $G/H= \{gH \mid g\in G\}$ and $H\backslash G= \{Hg \mid g\in G\}$.

Let $\ms{G}$ be a group operad. If $\ms{H}$ is a sub group operad of $\ms{G}$, then the $\gamma$ of $\ms{G}$ induces a $\gamma$ on $\ms{G}/ \ms{H}= \{G_k/H_k\}_{k\geq 0}$,
$$\xymatrix{
  G_k \times G_{m_1} \times \cdots \times G_{m_k} \ar[d] \ar[r]^-{\gamma}  & G_m \ar[d]  \\
  G_k/H_k \times G_{m_1}/ H_{m_1} \times \cdots \times G_{m_k}/ H_{m_k} \ar@{-->}[r]  & G_m/ H_m       }$$
since
$$\gamma (gh; g_1h_1, \ldots)= \gamma (g;g_1, \ldots) \gamma (h; h_{g^{-1}(1)}, \ldots) \in \gamma (g;g_1, \ldots) H_m.$$
$G_k$ acts on $G_k/H_k$ by $g\cdot aH_k= (ga)H_k$. If $\ms{H}$ is normal and noncrossed, then $G_k\to S_k$ factors through $G_k/H_k\to S_k$ and $\ms{G}/ \ms{H}$ is a group operad. If $\ms{H}$ is normal and crossed, there is no natural nontrivial $G_k/H_k\to S_k$; even if we let $G_k/H_k\to S_k$ be trivial, $\ms{G}/ \ms{H}$ is still not a group operad since $\gamma$ is not a homomorphism.

If $\ms{H}$ is a non-crossed sub group operad, the $\gamma$ of $\ms{G}$ also induces a $\gamma$ on $\ms{H} \backslash \ms{G}= \{H_k \backslash G_k\}_{k\geq 0}$,
$$\xymatrix{
  G_k \times G_{m_1} \times \cdots \times G_{m_k} \ar[d] \ar[r]^-{\gamma}  & G_m \ar[d]  \\
  H_k \backslash G_k \times H_{m_1} \backslash G_{m_1} \times \cdots \times H_{m_k} \backslash G_{m_k} \ar@{-->}[r]  & H_m  \backslash G_m       }$$
since $\gamma (hg; h_1g_1, \ldots)= \gamma (h;h_1, \ldots) \gamma (g; g_{h^{-1}(1)}, \ldots)= \gamma (h;h_1, \ldots) \gamma (g;g_1, \ldots) \in H_m \gamma (g;g_1, \ldots)$. If $\ms{H}$ is crossed, however, we may not have the above one, but the following one
$$\xymatrix{
  G_k \times G_{m_1} \times \cdots \times G_{m_k} \ar[d] \ar[r]^-{\gamma}  & G_m \ar[d]  \\
  G_k \times H_{m_1} \backslash G_{m_1} \times \cdots \times H_{m_k} \backslash G_{m_k} \ar@{-->}[r]  & H_m  \backslash G_m  }$$

\subsection{Simplicial Structure of Group Operads}
It is well known that both $\ms{S}$ and $\ms{B}$ are crossed simplicial groups. In fact, the operad structure makes a non-crossed group operad a simplicial group and a crossed group operad a crossed simplicial group.

Let $[n]= \{1, \ldots, n\}$ for $n\geq 1$. Define for $1\leq i\leq n$, $d^i: [n-1] \to [n]$, $j\mapsto j$ if $j<i$ and $j\mapsto j+1$ if $j\geq i$, and for $1\leq i\leq n-1$, $s^i: [n] \to [n-1]$, $j\mapsto j$ if $j\leq i$ and $j\mapsto j-1$ if $j>i$.

\begin{defn}
  A \textbf{crossed simplicial group} is a simplicial set $\{G_n\}_{n\geq 1}$ where all $G_n$ are groups, together with group homomorphisms $\pi_n: G_n \to S_n$, $n\geq 1$, such that
  \begin{itemize}
    \item[i)] $d_i(ab)= (d_ia) (d_{a(i)}b)$, $s_i(ab)= (s_ia) (s_{a(i)}b)$, and
    \item[ii)] the following diagrams are commutative:
        $$\begin{diagram}
          [n-1]  &\rTo^{d^i}  &[n]-\{i\} \\
          \dTo<{d_ia}  &  &\dTo>a \\
          [n-1]  &\rTo^{d^{a(i)}}  &[n]-\{a(i)\},
        \end{diagram} \qquad \qquad
        \begin{diagram}
          [n+1]  &\rTo^{s^i}  &[n] \\
          \dTo<{s_ia}  &  &\dTo>a \\
          [n+1]  &\rTo^{s^{a(i)}}  &[n].
        \end{diagram}$$
  \end{itemize}
  A crossed $\Delta$-group is defined in the same way but without degeneracies $s_i$.
\end{defn}

Note that a group operad $\ms{G}= \{G_n\}_{n\geq 0}$ is an operad with a strict basepoint. Thus $\ms{G}$ is a simplicial set with
\begin{align*}
  d_i= \gamma (-; e_1^{(i-1)}, e_0, e_1^{(n+1-i)}) & : G_{n+1} \to G_{n}, \quad 1\leq i\leq n+1, \\
  s_i= \gamma (-; e_1^{(i-1)}, e_2, e_1^{(n-i)}) & : G_n \to G_{n+1}, \quad 1\leq i\leq n.
\end{align*}

\begin{lem}\label{lem:pi_group-operad_simplicial-map}
  For a group operad $\ms{G}= \{G_n\}_{n\geq 0}$ with $\pi: \ms{G}\to \ms{S}$, $\pi: \ms{G}\to \ms{S}$ is a simplicial map and $d_i(ab)= (d_ia) (d_{a(i)}b)$ and $s_i(ab)= (s_ia) (s_{a(i)}b)$.
\end{lem}

\begin{proof}
  $\pi$ is a simplicial map as it commutes with $\gamma$. Let $c_j=e_1$ for $j\neq i$ and $c_i=e_0$. Since $\gamma$ is a crossed homomorphism,
  \begin{align*}
    d_i(ab) & = \gamma (ab; e_1^{(i-1)}, e_0, e_1^{(n-i)})= \gamma (ab; c_1, \ldots, c_n)\\
    &= \gamma (a; c_1, \ldots, c_n) \gamma (b; c_{a^{-1}(1)}, \ldots, c_{a^{-1}(a(i))}, \ldots, c_{a^{-1}(n)}) \\
    &= (d_ia) \gamma (b; e_1^{(a(i)-1)}, e_0, e_1, \ldots)= (d_ia) (d_{a(i)}b).
  \end{align*}
  Similarly $s_i(ab)= \gamma (ab; e_1^{(i-1)}, e_2, e_1^{(n-i)})= (s_ia) \gamma (b; e_1^{(a(i)-1)}, e_2, \ldots)= (s_ia) (s_{a(i)}b)$.
\end{proof}

\begin{prop}
  Any non-crossed group operad is a simplicial group and any crossed group operad is a crossed simplicial group. Any morphism of non-crossed group operads is a morphism of simplicial groups and any morphism of crossed group operads is a morphism of crossed simplicial groups. \qed
\end{prop}

\begin{rem}
  The definition of a crossed simplicial group ($\Delta$-group) is defined in this way to make it consistent with the operad structure of a group operad defined in this paper. It is different from, but equivalent to the one given by Proposition 1.7 of \cite{Fie-Lod:1991:CSGTAH}. It should be noted that the $\{\pi_n\}_{n\geq 1}$ of a crossed simplicial group need not be a simplicial map (Note that Lemma \ref{lem:pi_group-operad_simplicial-map} is about group operads, not about crossed simplicial groups; two examples of crossed simplicial groups which are not group operads are given at the end of this subsection). From the first diagram in ii), $\{\pi_n\}_{n\geq 1}$ commutes with faces $d_i$ and thus is a $\Delta$-map. However, from the second diagram, we have
$$\{(s_ia)(i), (s_ia)(i+1)\}= \{a(i), a(i)+1\},$$
but may not have $(s_ia)(i)= a(i)$, $(s_ia)(i+1)= a(i)+1$. Namely $\{\pi_n\}_{n\geq 1}$ needn't commute with degeneracies $s_i$. For instance, the sequence of hyperoctahedral groups satisfies that $(s_ia)(i)= a(i)+1$ and $(s_ia)(i+1)= a(i)$ for certain $a$ as discussed in Section 3 of \cite{Fie-Lod:1991:CSGTAH}. Crossed simplicial groups are also defined in Subsection 3.1 of \cite{Ber-Coh-Won-Wu:2006:CBHG} where $\{\pi_n\}_{n\geq 1}$ is furthermore required to commute with degeneracies. Thus the definition given in \cite{Fie-Lod:1991:CSGTAH} includes the one given in \cite{Ber-Coh-Won-Wu:2006:CBHG}, but the latter does not include the former. Therefore the claim right after the definition of a crossed simplicial group in \cite{Ber-Coh-Won-Wu:2006:CBHG} that they are equivalent is wrong.
\end{rem}

We next construct a wreath product of crossed $\Delta$-groups and of crossed simplicial groups following a wreath product of the symmetric groups given in \cite{Fie-Lod:1991:CSGTAH}. Let $A$ be a group and $\phi: S_n\to \Aut (A^n)$, $\sigma \mapsto \phi_{\sigma}$ where $\phi_{\sigma} (a_1, \ldots, a_n)= (a_{\sigma (1)}, \ldots, a_{\sigma (n)})$. Let $\ms{H}= \{H_n\}_{n\geq 0}$ be a crossed $\Delta$-group with $\phi: H_n \to S_n \to \Aut (A^n)$. Recall that the wreath product $A\wr H_n$ of $H_n$ by $A$ is defined as the semidirect product $A^n \rtimes H_n$. Define the wreath product of $\ms{H}$ by $A$ as
$$A\wr \ms{H}= \{A\wr H_n\}_{n\geq 0},$$
with $d_i (a_1, \ldots, a_n; h)= (a_1, \ldots, \hat{a}_i, \ldots, a_n; d_ih)$ and $\pi: A\wr H_n= A^n \rtimes H_n \onto H_n \to S_n$. Clearly $A\wr \ms{H}= \{A^n\}_{n\geq 0} \times \ms{H}$ as $\Delta$-sets. If $\ms{H}$ is moreover a crossed simplicial group, define $s_i (a,\ldots, a_n; h)= (a_1, \ldots, a_i, a_i, \ldots, a_n; s_i h)$; then $A\wr \ms{H}= \{A^n\}_{n\geq 0} \times \ms{H}$ as simplicial sets.

\begin{prop}
  $A\wr \ms{H}$ is a crossed $\Delta$-group if $\ms{H}$ is a crossed $\Delta$-group, and a crossed simplicial group if $\ms{H}$ is a crossed simplicial group.
\end{prop}

\begin{proof}
  If $\ms{H}$ is a crossed $\Delta$-group, it suffices to verify that for $h\in H_n$ and $\bar{a}= (a_1, \ldots, a_n)\in A^n$,
  $$d_i ((1;h) (\bar{a};1))= (d_i(1;h)) (d_{(1;h)(i)} (\bar{a};1))= (1;d_ih) (d_{h(i)} \bar{a};1).$$
  Note $d_i ((1;h) (\bar{a};1))= d_i (\phi_h (\bar{a});h)$, $(1;d_ih) (d_{h(i)} \bar{a};1)= (\phi_{d_ih} (d_{h(i)} \bar{a}); d_ih)$. Thus it is sufficient to check $d_i \phi_h (\bar{a})= \phi_{d_ih} (d_{h(i)} \bar{a})$, namely
  $$d_i \phi_h (a_1, \ldots, a_n)= \phi_{d_ih} (a_1, \ldots, \hat{a}_{h(i)}, \ldots, a_n).$$
  Note $d_i \phi_h (a_1, \ldots, a_n)= (a_{h(1)}, \ldots, \hat{a}_{h(i)}, \ldots, a_{h(n)})$. Let $b_j=a_j$ for $j<h(i)$ and $b_j= a_{j+1}$ for $j\geq h(i)$. Then
  $$\phi_{d_ih} (a_1, \ldots, \hat{a}_{h(i)}, \ldots, a_n)= \phi_{d_ih} (b_1, \ldots, b_{n-1})= (b_{(d_ih)(1)}, \ldots, b_{(d_ih)(n-1)}).$$
  If $j<i$ and $h(j)< h(i)$, then $(d_ih)(j)= h(j)$, $b_{(d_ih)(j)}= a_{h(j)}$; if $j<i$ and $h(j)>h(i)$, then $(d_ih)(j)= h(j)-1$, $b_{(d_ih)(j)}= b_{h(j)-1}= a_{h(j)}$. Thus
  $$(a_{h(1)}, \ldots, a_{h(i-1)})= (b_{(d_ih)(1)}, \ldots, b_{(d_ih)(i-1)}).$$
  If $j>i$ and $h(j)< h(i)$, then $(d_ih)(j-1)=h(j)$, $b_{(d_ih)(j-1)}= b_{h(j)}= a_{h(j)}$; if $j>i$ and $h(j)>h(i)$, then $(d_ih)(j-1)=h(j)-1$, $b_{(d_ih)(j-1)}= b_{h(j)-1}= a_{h(j)}$. Thus
  $$(a_{h(i+1)}, \ldots, a_{h(n)})= (b_{(d_ih)(i)}, \ldots, b_{(d_ih)(n-1)}).$$
  Hence $d_i \phi_h (\bar{a})= \phi_{d_ih} (d_{h(i)} \bar{a})$.

  If $\ms{H}$ is a crossed simplicial group, $s_i ((1;h) (\bar{a};1))= (s_i(1;h)) (s_{(1;h)(i)} (\bar{a};1))$ can be verified similarly.
\end{proof}

Similar to the semidirect product of groups, we have

\begin{prop}
  If $\ms{G}= A\wr \ms{H}$ where $\ms{H}$ is a crossed $\Delta$-group (resp. crossed simplicial group), then there is a split short exact sequence of crossed $\Delta$-groups (resp. crossed simplicial groups)
$$1\to \{A^n\}_{n\geq 0}\to \ms{G}\to \ms{H} \to 1$$
with the canonical morphism of crossed $\Delta$-groups (resp. crossed simplicial groups) $\ms{H} \to \ms{G}$ as the section. Conversely, if the above short exact sequence splits with a morphism of crossed $\Delta$-groups (resp. crossed simplicial groups) $\ms{H} \to \ms{G}$ as a section, then $\ms{G}\cong A\wr \ms{H}$ as crossed $\Delta$-groups (resp. crossed simplicial groups). \qed
\end{prop}

\begin{example}
  For the ribbon braid groups operad $\ms{R}$, its degeneracy $s_i$ of $\ms{R}$ is obtained by cutting the $i$th ribbon along the middle into two subribbons; thus if the $i$th ribbon is twisted, then the two subribbons are twisted with each other. In addition, we have the following split short exact sequence of crossed $\Delta$-groups
  $$1\to \{\Z^n\}_{n\geq 0}\to \ms{R}\to \ms{B}\to 1,$$
  namely $\ms{R}\cong \Z\wr \ms{B}$ as crossed $\Delta$-groups. It should be mentioned that this is not an isomorphism of simplicial sets thus not of crossed simplicial groups. In addition, we have the following obvious short exact sequence of crossed simplicial groups
  $$1\to \ms{P}\to \ms{B} \to \ms{S}\to 1$$
  which is not split.
\end{example}

We next give two natural examples which are crossed simplicial groups but not group operads.

\begin{example}
  The sequence of hyperoctahedral groups $\wt{\ms{S}}= \{\wt{S}_n\}$ with $\wt{S}_n= (\Z/2)\wr S_n$ is naturally a crossed simplicial group (cf. \cite{Fie-Lod:1991:CSGTAH}, Theorem 3.3), but not a group operad. Let $\tau$ be the generator of $\Z/2$. First it is obvious to define $\gamma (\tau;e_k)$ as the half twist. Then there are two natural ways to define $\gamma (\tau;a)$: $\gamma (\tau;a):= \gamma (\tau;e_k)a$ or $\gamma (\tau;a)= a\gamma (\tau;e_k)$; but they do not coincide since $\gamma (\tau;e_k)$ is not in the center for $k\geq 3$, so that they do not behave well. In addition, we have the following split short exact sequence of crossed $\Delta$-groups
$$1\to \{(\Z/2)^n\}_{n\geq 0}\to \wt{\ms{S}}\to \ms{S}\to 1,$$
namely $\wt{\ms{S}}\cong (\Z/2)\wr \ms{S}$ as crossed $\Delta$-groups. However this is not an isomorphism of simplicial sets (cf. \cite{Fie-Lod:1991:CSGTAH}, Remark 3.11) thus not of crossed simplicial groups.
\end{example}

\begin{example}
  Let $R'_n$ be the braid group on $n$ ribbons with each ribbon twisted a multiple of the half twist. $\ms{R}'= \{R'_n\}_{n\geq 0}$ is naturally a crossed simplicial group but not a group operad. The reason is the same as the one for the sequence of hyperoctahedral groups; namely the half twist is not in the center so that the two natural ways of defining $\gamma$ do not coincide. In addition, we have the following short exact sequence of crossed $\Delta$-groups
$$1\to \{(2\Z)^n\}_{n\geq 0} \times \ms{P}\to \ms{R}'\to \wt{\ms{S}}\to 1$$
which is not split.
\end{example}

\subsection{Topological and Simplicial $\ms{G}$-Operads}
Compared to general operads, group operads play a special role like groups. Namely, we can talk about actions of group operads on operads just like actions of groups on spaces. We call an operad with an action of a group operad $\ms{G}$ a $\ms{G}$-operad. Then a nonsymmetric operad is an operad with an action of the trivial group operad and a symmetric operad is an operad with an action of the symmetric groups operad. The theory of symmetric operads can be generalized to $\ms{G}$-operads. Here we shall only deal with topological and simplicial $\ms{G}$-operads.

\begin{defn}\label{defn:operad}
  An \textbf{action} of a group operad on a topological operad $\ms{C}= \{\ms{C}(n)\}_{n\geq 0}$ is a sequence of left actions of $G_n$ on $\ms{C}(n)$ satisfying the following \emph{equivariance} property: for $a\in \ms{C}(k)$, $b_i\in \ms{C}(m_i)$ and $\sigma \in G_k$, $\tau_i\in G_{m_i}$,
$$\gamma (\sigma a; \tau_1 b_1, \ldots, \tau_k b_k)= \gamma (\sigma; \tau_1, \ldots, \tau_k) \gamma (a; b_{\sigma^{-1}(1)}, \ldots, b_{\sigma^{-1}(k)}).$$
A \textbf{topological} $\ms{G}$\textbf{-operad} is an operad with an action of $\ms{G}$. A \textbf{morphism} $\psi: \ms{C}\to \ms{C}'$ of $\ms{G}$-operads is a sequence of $G_n$-equivariant maps $\psi_n: \ms{C}(n)\to \ms{C}'(n)$ such that $\psi_1(1)= 1$ and
$$\psi_m (\gamma (a; b_1, \ldots, b_k))= \gamma' (\psi_k(a); \psi_{m_1}(b_1), \ldots, \psi_{m_k}(b_k)).$$
The action of $\ms{G}$ on $\ms{C}$ is called a covering action and $\ms{C}$ is called a \textbf{covering $\ms{G}$-operad} if the action of $G_n$ on $\ms{C}(n)$ is a covering action for each $n$. $\ms{C}$ is called a (topological) \textbf{universal $\ms{G}$-operad} if $\ms{C}$ is a covering $\ms{G}$-operad and each $\ms{C}(n)$ is contractible.

  Let $\ms{A}(n)$ be a subspace of $\ms{C}(n)$ for each $n$ and $1\in \ms{A}(1)$. If $\ms{A}= \{\ms{A}(n)\}$ is closed under $\gamma$ and the action of $\ms{G}$, then $\ms{A}$ is called a (topological) \textbf{$\ms{G}$-suboperad} of $\ms{C}$.
\end{defn}

For a $\ms{G}$-operad $\ms{C}$, a $\ms{C}$-space can be defined just as the case $\ms{G}= \ms{S}$ (\cite{May:1972:GILS}, Section 1) but replacing $S_k$ by $G_k$ for all $k$. When dealing with a $\ms{G}$-operad $\ms{C}$ and $\ms{C}$-spaces, one should pay attention to the group operad $\ms{G}$ as an operad may admit actions of different group operads which may cause different results. For example, if the trivial group operad $\ms{J}$ is thought of as a $\ms{J}$-operad, i.e. a nonsymmetric operad, then a $\ms{J}$-space is a topological monoid; but if it is thought of as a $\ms{S}$-operad, i.e. a symmetric operad, then a $\ms{J}$-space is a commutative topological monoid.

Note that a group operad $\ms{G}$ is itself a $\ms{G}$-operad. Any $\ms{G}$-operad can also be regarded as an $\ms{H}$-operad for any sub group operad $\ms{H}\leq \ms{G}$. In particular, any $\ms{G}$-operad can be regarded as a nonsymmetric operad, but a $\ms{G}$-operad usually can not be regarded as a symmetric operad even if $\ms{G}$ is crossed.

We say a topological $\ms{G}$-operad $\ms{C}$ is path-connected (resp. locally path-connected, semilocally simply-connected, etc.) if $\ms{C}(k)$ is path-connected (resp. locally path-connected, semilocally simply-connected (\cite{Hatcher:2002:AT}, page 63), etc.) for each $k$. We also say $\ms{C}$ is $K(\pi,n)$ if $\ms{C}(k)$ is $K(\pi,n)$ for each $k$ ($\pi$ not necessary to be the same for different $k$).

\begin{prop}
  For a topological operad $\ms{C}$, $\ms{C}(1)$ is an associative $H$-space with a strict identity.
\end{prop}

\begin{proof}
  $\gamma: \ms{C}(1)\times \ms{C}(1)\to \ms{C}(1)$ gives a product on $\ms{C}(1)$. This product is associative and has a strict identity follows from the associativity and the unitality of $\gamma$, respectively.
\end{proof}

In addition, $\ms{C}(1)$ acts on $\ms{C}(k)$ for $k\geq 1$ just like a group action.

\begin{rem}\label{rem:extending_G-space_G-operad}
  For a space $X$ with a free action of $S_2$ and $X\stackrel{S_2}{\simeq} S^{n-1}$, Fiedorowicz introduces a construction (cf.\cite{Fiedorowicz:1999:CEnO}, Theorem 2) to extend $X$ to an $E_n$ operad (i.e. equivalent to the little $n$-cubes operad $\ms{C}_n$). Fiedorowicz's construction has been generalized to any spaces with actions of groups. For any group $G$ with a homomorphism $G\to S_2$ and any space $X$ with an action $\rho: G\times X\to X$, there is a construction extending $X$ to a $\ms{F}(G,\pi)$-operad $\ms{F}(X,\rho)$ with $\ms{F}(X,\rho)(1)=\{1\}$ and $\ms{F}(X,\rho)(2)=X$, where $\ms{F}(G,\pi)$ is a group operad extended from $G$ mentioned in Remark \ref{rem:extending_group_group-operad}. This construction has the property that, if the action of $G$ on $X$ is free, then so is the action of $\ms{F}(G,\pi)(k)$ on $\ms{F}(X,\rho)(k)$. Moreover Fiedorowicz's construction of $E_n$ operads from a space $X\stackrel{S_2}{\simeq} S^{n-1}$ is a special case. Thus the construction $\ms{F}(X,\rho)$ seems interesting and provides countless examples of $\ms{G}$-operads. In addition it may be helpful to find appropriate models of $E_n$ operads or other interesting operads for particular purposes. Details of this construction will appear elsewhere.
\end{rem}

We next define equivalence between topological $\ms{G}$-operads.

\begin{defn}
  A morphism $\psi: \ms{C}\to \ms{C}'$ of topological $\ms{G}$-operads is called an \textbf{equivalence}, if each $\psi: \ms{C}(k)\to \ms{C}'(k)$ is (1) a $G_k$-equivariant homotopy equivalence, or (2) a homotopy equivalence and the actions of $G_k$ on $\ms{C}(k)$, $\ms{C}'(k)$ are covering actions for all $k$. Two topological $\ms{G}$-operads are equivalent if there is a chain of equivalences connecting them.
\end{defn}

This definition is a slight modification of P. May's Definition 3.3 in \cite{May:1972:GILS}. It should be mentioned that an equivalence $\psi: \ms{C}\to \ms{C}'$ need not have an inverse morphism.

A simplicial $\ms{G}$-operad is defined in the obvious way, so we shall omit the details. Note that a simplicial $\ms{G}$-operad $\ms{C}$ is a bisimplicial set; within the simplicial structure of $\ms{C}(n)$, the action of $G_n$ is a simplicial action, but within the simplicial structure of $\ms{C}$ induced by $\gamma$, the action of $\ms{G}$ is a crossed simplicial action.

\begin{prop}
  The geometric realization of a simplicial $\ms{G}$-operad is a topological $\ms{G}$-operad. \qed
\end{prop}

The proof is routine but tedious.

$\ms{G}$-operads can be converted into nonsymmetric or symmetric operads by taking quotients. Here we shall only consider quotients of topological operads and just mention that the following discussion and results are also valid for simplicial operads.

Let $\ms{C}$ a topological $\ms{G}$-operad. Note that $G_n$ acts on $\ms{C}(n)$ on the left. For $H_n\leq G_n$, it is natural to use $H_n \backslash \ms{C}(n)$ to denote the quotient of $\ms{C}(n)$ modulo the left action of $H_n$, but we would like to use $\ms{C}(n)/ H_n$ instead even though the latter may cause ambiguity in case $\ms{C}= \ms{G}$. Namely $G_n/H_n$ denotes $\{gH_n \mid g\in G_n\}$ if $\ms{G}$ is regarded as a group operad, and denotes $H_n \backslash G_n= \{H_ng \mid g\in G_n\}$ if $\ms{G}$ is regarded as a $\ms{G}$-operad.

\begin{prop}
  If $\ms{H}$ is a non-crossed sub group operad of $\ms{G}$, then $\ms{C}/\ms{H}$ is a nonsymmetric operad and a morphism $\phi: \ms{C}\to \ms{C}'$ of $\ms{G}$-operads induces a morphism $\phi: \ms{C}/\ms{H} \to \ms{C}'/\ms{H}$ of nonsymmetric operads.
\end{prop}

\begin{proof}
  $\gamma: \ms{C}(k) \times \ms{C}(m_1) \times \cdots \times \ms{C}(m_k) \to \ms{C}(m)$ induces
  $$\gamma: \ms{C}(k)/H_k \times \ms{C}(m_1)/ H_{m_1} \times \cdots \times \ms{C}(m_k)/ H_{m_k} \to \ms{C}(m)/ H_m$$
  since $\gamma (ha; h_1a_1, \ldots, h_ka_k)= \gamma (h;h_1, \ldots, h_k) \gamma (a;a_1, \ldots, a_k)$. It is then easy to check that $\ms{C}/\ms{H}$ is a nonsymmetric operad and the rest of the assertion.
\end{proof}

\begin{prop}
  If $\ms{H}$ is a non-crossed normal sub group operad of $\ms{G}$, then $\ms{C}/\ms{H}$ is a $\ms{G}/ \ms{H}$-operad, $(\ms{C}/ \ms{H})/ (\ms{G}/ \ms{H})= \ms{C}/ \ms{G}$, and a morphism $\phi: \ms{C}\to \ms{C}'$ of $\ms{G}$-operads induces a morphism $\phi: \ms{C}/\ms{H} \to \ms{C}'/\ms{H}$ of $\ms{G}/ \ms{H}$-operads.
\end{prop}

\begin{proof}
  If $\ms{H}$ is normal, $\ms{G}/ \ms{H}$ naturally acts on $\ms{C}/\ms{H}$ by $gH_k \cdot H_ka= H_k(ga)$. It is then easy to check that $\ms{C}/\ms{H}$ is a $\ms{G}/ \ms{H}$-operad and the rest of the assertion.
\end{proof}

If $\ms{H}$ is not normal, $\ms{G}$ does not naturally act on $\ms{C}/ \ms{H}$ since there is no natural definition for $g\cdot H_ka$.

\begin{cor}
  Let $\psi: \ms{G}\onto \ms{G}'$ be an epimorphism of group operads. If $\ms{C}$ is a $\ms{G}$-operad, then $\ms{C}/ \Ker \psi$ is a $\ms{G}'$-operad. In particular, if $\ms{G}$ is non-crossed, then $\ms{C}/ \ms{G}$ is a nonsymmetric operad; if $\ms{G}$ is crossed with $\pi: \ms{G}\onto \ms{S}$, then $\ms{C}/ \Ker \pi$ is a symmetric operad. \qed
\end{cor}

\section{Homotopy Groups of Topological Operads}
It is known that the homology of a nonsymmetric (resp. symmetric) topological operad is a nonsymmetric (resp. symmetric) algebraic operad. It seems, however, that operad structure on the homotopy groups of topological operads has not been considered carefully, possibly because of the lack of motivation or the difficulty of dealing with basepoints. If an operad $\ms{C}$ has a strict basepoint, then the $\gamma$ of $\ms{C}$ obviously induces a natural $\gamma$ on homotopy groups. It is, however, impossible to find a strict basepoint for some important topological operads, such as the little $n$-cubes operads. The lack of a strict basepoint makes the situation complicated. Nevertheless, once the difficulty of basepoints is overcome, it turns out that a strict basepoint is indeed not necessary and that the homotopy groups of topological operads provide examples of group operads and discrete operads with actions of group operads. The basic idea of the proof is similar to the one for that the fundamental groups of configuration spaces of manifolds are (crossed) simplicial groups \cite{Ber-Coh-Won-Wu:2006:CBHG}, but the situation here is more complicated. We are dealing with operad structure, not only (crossed) simplicial group structure.

To deal with the homotopy groups of topological operads, we introduce a nonsymmetric topological operad which will play the role of a good basepoint for topological operads. Consider the little 1-cubes (intervals) operad $\ms{C}_1$. Let $c_0=*\in \ms{C}_1(0)=*$ and $c_k=([0,\frac{1}{k}], [\frac{1}{k}, \frac{2}{k}], \ldots, [\frac{k-1}{k},1]) \in \ms{C}_1(k)$. Let $\ms{C}_1(k)_0 \subseteq \ms{C}_1(k)$ be the path-connected component of $c_k$. Then $(\ms{C}_1)_0= \{\ms{C}_1(k)_0\}_{k\geq 0}$ is naturally a nonsymmetric suboperad of $\ms{C}_1$. Note that for $c,c'\in \ms{C}_1(k)_0$ and $t\in [0,1]$, $(1-t)c+tc'$ can be canonically defined and $(1-t)c+tc'\in \ms{C}_1(k)_0$. Call $\delta: I\to \ms{C}_1(k)_0$, $t\mapsto (1-t)c+tc'$ the linear path from $c$ to $c'$, denoted $c\xra{\delta} c'$.

\begin{lem}
  Any two paths $f$ and $g: I\to \ms{C}_1(k)_0$ with $f(0)=g(0)$ and $f(1)=g(1)$ are linearly homotopic rel $\{0,1\}$, namely, $H(s,t)= (1-t)f(s)+ tg(s)$ is a linear homotopy from $f$ to $g$. \qed
\end{lem}

Suppose $\ms{C}$ is a topological $\ms{G}$-operad admitting a morphism of nonsymmetric operads $\eta: (\ms{C}_1)_0\to \ms{C}$. For instance, if $\ms{C}$ has a strict basepoint, then $\eta: (\ms{C}_1)_0 \to *\into \ms{C}$; for $\ms{C}_n$, there is a canonical inclusion $\eta: (\ms{C}_1)_0 \into \ms{C}_n$. Call $\{\eta(c_k)\}_{k\geq 0}$ a \textbf{good basepoint} of $\ms{C}$ and also say $\ms{C}$ is \textbf{well pointed}. For two well pointed topological operads $\ms{C}$ with $\eta: (\ms{C}_1)_0 \to \ms{C}$ and $\ms{C}'$ with $\eta': (\ms{C}_1)_0 \to \ms{C}'$, a morphism $\psi: \ms{C}\to \ms{C}'$ is called a morphism of well pointed topological operads if $\eta'= \psi\circ \eta$.

Clearly $c\xra{\delta} c'$ in $\ms{C}_1(k)_0$ gives a path $\eta \circ \delta$ from $\eta(c)$ to $\eta(c')$ in $\ms{C}(k)$, denoted $\eta(c) \xra{\eta \circ \delta} \eta(c')$ and called the linear path from $\eta(c)$ to $\eta(c')$ for convenience. We shall omit $\eta \circ \delta$ and simply write $\eta(c) \to \eta(c')$ from now on.

\begin{lem}
  Suppose $f,g: I\to \ms{C}(k)_0$ are two paths with $f(0)=g(0)$ and $f(1)=g(1)$. If they can be pulled back to two paths in $\ms{C}_1(k)_0$, i.e., if there are $f',g': I\to \ms{C}_1(k)_0$ with $f'(0)=g'(0)$ and $f'(1)=g'(1)$ such that $f= \eta\circ f'$ and $g= \eta\circ g'$, then $f\simeq g$ rel $\{0,1\}$. \qed
\end{lem}

For two paths $f,g: I\to X$ with $f(1)= g(0)$, let $f\cdot g$ be the usual product of $f$ and $g$. Throughout this section, assume that $\ms{C}$ is a path-connected topological $\ms{G}$-operad with a good basepoint. Let $e_k= \eta(c_k)$. The linear path $e_m \to \gamma (e_k; e_{m_1}, \ldots, e_{m_k})$ and its inverse will be helpful when changing basepoint. Let $e_k$ also denote the constant loop $S^1 \to e_k\into \ms{C}(k)$ (or $I \to e_k\into \ms{C}(k)$).

\subsection{Fundamental Groups of Nonsymmetric Operads}
Let $\ms{C}$ be a nonsymmetric topological operad. Define
$$\gamma: \pi_1 (\ms{C}(k), e_k) \times \prod_{i=1}^k \pi_1 (\ms{C}(m_i), e_{m_i}) \xra{\gamma_*} \pi_1 (\ms{C}(m), \gamma (e_k; e_{m_1}, \ldots)) \xra{\delta_*} \pi_1 (\ms{C}(m), e_m)$$
where $\delta_*$ is the isomorphism induced by $e_m \xra{\delta} \gamma (e_k; e_{m_1}, \ldots)$. Namely, for $[f]\in \pi_1 (\ms{C}(k), e_k)$ and $[f_i]\in \pi_1 (\ms{C}(m_i), e_{m_i})$,
$$\gamma ([f]; [g_1], \ldots, [g_k])= [\delta] [\gamma (f; g_1, \ldots, g_k)] [\delta]^{-1}.$$
Thus $\gamma ([f]; [g_1], \ldots, [g_k])$ is represented by the following path
$$e_m \xra{\delta} \gamma (e_k; e_{m_1}, \ldots) \xra{\gamma (f; g_1, \ldots)} \gamma (e_k; e_{m_1}, \ldots) \xra{\delta^{-1}} e_m.$$
$\delta$ and $\delta^{-1}$ will be omitted from now on.

\begin{thm}
  If $\ms{C}$ is a path-connected nonsymmetric operad with a good basepoint, then $\pi_1\ms{C}$ with $\gamma$ defined above is a non-crossed group operad.
\end{thm}

\begin{proof}
  Let $[f], [f']\in \pi_1 (\ms{C}(k), e_k)$, $[g_i], [g'_i]\in \pi_1 (\ms{C}(m_i), e_{m_i})$, $[h_j]\in \pi_1 (\ms{C}(n_j), e_{n_j})$.

  Unitality. $\gamma (e_1; f)(z)= \gamma (e_1; f(z))= f(z)$, $\gamma (f; e_1^{(k)}) (z)= \gamma (f(z); e_1^{(k)})= f(z)$ for $z\in S^1$. Thus $\gamma ([e_1];[f])= [f]$ and $\gamma ([f]; [e_1]^{(k)})= [f]$.

  Next we check that $\gamma$ is a homomorphism. $\gamma ([f][f']; [g_1][g'_1], \ldots, [g_k][g'_k])$ is represented by the following path
  $$e_m \to \gamma (e_k; e_{m_1}, \ldots) \xra{\gamma (f; g_1, \ldots)} \gamma (e_k; e_{m_1}, \ldots) \xra{\gamma (f'; g'_1, \ldots)} \gamma (e_k; e_{m_1}, \ldots) \to e_m,$$
  while $\gamma ([f]; [g_1], \ldots, [g_k]) \gamma ([f']; [g'_1], \ldots, [g'_k])$ is represented by
  \begin{align*}
    e_m & \to \gamma (e_k; e_{m_1}, \ldots) \xra{\gamma (f; g_1, \ldots)} \gamma (e_k; e_{m_1}, \ldots) \to e_m \\
    & \to \gamma (e_k; e_{m_1}, \ldots) \xra{\gamma (f'; g'_1, \ldots)} \gamma (e_k; e_{m_1}, \ldots) \to e_m
  \end{align*}
  So $\gamma ([f][f']; [g_1][g'_1], \ldots, [g_k][g'_k])= \gamma ([f]; [g_1], \ldots, [g_k]) \gamma ([f']; [g'_1], \ldots, [g'_k])$.

  Associativity. $\gamma (\gamma ([f]; [g_1], \ldots, [g_k]); [h_1], \ldots, [h_m])$ is represented by
  \begin{align*}
    e_n & \to \gamma (e_m; e_{n_1}, \ldots) \to \gamma (\gamma (e_k; e_{m_1}, \ldots); e_{n_1}, \ldots) \\
    & \xra{\gamma (\gamma (f; g_1, \ldots); h_1, \ldots)} \gamma (\gamma (e_k; e_{m_1}, \ldots); e_{n_1}, \ldots) \to \gamma (e_m; e_{n_1}, \ldots) \to e_n
  \end{align*}
  which is of the form $\delta_1 \cdot \gamma (\gamma (f; g_1, \ldots); h_1, \ldots) \cdot \delta_2$ where $\delta_1, \delta_2$ are two linear paths. $\gamma ([f]; \gamma ([g_1]; [h_1], \ldots), \ldots, \gamma ([g_k]; \ldots, [h_m]))$ is represented by
  \begin{align*}
    e_n & \to \gamma (e_k; e_{n_1+ \cdots+ n_{m_1}}, \ldots) \to \gamma (e_k; \gamma (e_{m_1}; e_{n_1}, \ldots), \ldots) \\
    & \xra{\gamma (f; \gamma (g_1; h_1, \ldots), \ldots)} \gamma (e_k; \gamma (e_{m_1}; e_{n_1}, \ldots), \ldots) \to \gamma (e_k; e_{n_1+ \cdots+ n_{m_1}}, \ldots) \to e_n
  \end{align*}
  which is of the form $\delta'_1 \cdot \gamma (f; \gamma (g_1; h_1, \ldots), \ldots) \cdot \delta'_2$ where $\delta'_1, \delta'_2$ are two linear paths. From the associativity of the $\gamma$ of $\ms{C}$,
  $$\gamma (\gamma (f; g_1, \ldots, g_k); h_1, \ldots, h_m)= \gamma (f; \gamma (g_1; h_1, \ldots), \ldots, \gamma (g_k; \ldots, h_m)).$$
  Moreover, $\delta_1$ and $\delta'_1$, $\delta_2$ and $\delta'_2$ are homotopic rel $\{0,1\}$ since they can be pulled back to $\ms{C}_1(k)_0$. Thus the associativity holds.
\end{proof}

\subsection{Fundamental Groups of Symmetric Operads}
Given a topological space $X$ and two subspaces $A,B$, let
$$\pi_1 (X;A,B)= \{[f] \mid f:I\to X, f(0)\in A, f(1)\in B\}.$$

Let $\ms{C}$ be a covering symmetric operad. Then
$$S_k\to \ms{C}(k)\onto \ms{C}(k)/S_k$$
is a fibration. Thus there is a long exact sequence
$$\cdots\to \pi_lS_k\to \pi_l \ms{C}(k)\to \pi_l (\ms{C}(k)/S_k)\to \pi_{l-1}S_k\to \cdots.$$
Since $\pi_0S_k=S_k$ and $\pi_lS_k=1$ for $l>0$, there is a short exact sequence
$$1\to \pi_1\ms{C}(k)\to \pi_1 (\ms{C}(k)/S_k) \xra{\pi} S_k\to 1$$
and $\pi_l \ms{C}(k)\cong \pi_l (\ms{C}(k)/S_k)$ for $l>1$.

$\pi_1 (\ms{C}(k)/S_k, S_ka)$ can be identified with $\pi_1 (\ms{C}(k); a, S_ka)$ as follows. For any $a\in \ms{C}(k)$, the natural function
$$\pi_1 (\ms{C}(k); a, S_ka)\to \pi_1 (\ms{C}(k)/S_k, S_ka)$$
is a bijection. Notice that $[f]\in \pi_1 (\ms{C}(k)/S_k, S_ka)$ is lifted to a path from $a$ to $(\pi[f])a$. The multiplication in $\pi_1 (\ms{C}(k)/S_k, S_ka)$ induces a multiplication in $\pi_1 (\ms{C}(k); a, S_ka)$:
$$f*g:= f\cdot (\pi[f])g, \quad [f]*[g]:= [f*g]= [f]\cdot [(\pi[f])g]$$
for $[f],[g]\in \pi_1 (\ms{C}(k); a, S_ka)$, where
$$\pi: \pi_1 (\ms{C}(k); a, S_ka)\to \pi_1 (\ms{C}(k)/S_k, S_ka) \xra{\pi} S_k,$$
noting that $f\cdot (\pi[f])g$ is a path from $a$ to $\pi[f]a$ then to $\pi[f] (\pi[g]a)$. With this multiplication, $\pi_1 (\ms{C}(k); a, S_ka)\to \pi_1 (\ms{C}(k)/S_k, S_ka)$ is an isomorphism. Let $\pi f= \pi[f]\in S_k$ and $[f](i)= (\pi[f]) (i)$. Clearly $\pi ([f]*[g])= \pi[f] \pi[g]$ and $\pi (f*g)= \pi f \pi g$.

Define $\gamma$ to be the following composite
\begin{align*}
  \gamma: &\ \pi_1 (\ms{C}(k); e_k, S_ke_k) \times \prod_{i=1}^k \pi_1 (\ms{C}(m_i); e_{m_i}, S_{m_i} e_{m_i}) \\
  \to &\ \pi_1 (\ms{C}(m); e'_m, \{\gamma (\sigma; \tau_{m_1}, \ldots, \tau_{m_k}) e_m^{\sigma} \mid \sigma \in S_k, \tau_i\in S_{m_i}\}) \\
  \to &\ \pi_1 (\ms{C}(m); e_m, \{\gamma (\sigma; \tau_{m_1}, \ldots, \tau_{m_k}) e_m \mid \sigma \in S_k, \tau_i\in S_{m_i}\}) \\
  \into &\ \pi_1 (\ms{C}(m); e_m, S_m e_m),
\end{align*}
where $e'_m= \gamma (e_k; e_{m_1}, \ldots, e_{m_k})$ and $e_m^{\sigma}= \gamma (e_k; e_{m_{\sigma^{-1}(1)}}, \ldots, e_{m_{\sigma^{-1}(k)}})$; explicitly,
\begin{align*}
  \gamma ([f]; [g_1], \ldots, [g_k]) & = [\delta_1] [\gamma (f; g_1, \ldots, g_k)]* [(\delta_2)^{-1}]\\
  & = [\delta_1] [\gamma (f; g_1, \ldots, g_k)] [\gamma (\pi[f]; \pi[g_1], \ldots, \pi[g_k]) (\delta_2)^{-1}]
\end{align*}
for $[f]\in \pi_1 (\ms{C}(k); e_k, S_ke_k)$ and $[g_i]\in \pi_i (\ms{C}(m_i); e_{m_i}, S_{m_i} e_{m_i})$, where $e_m \xra{\delta_1} e'_m$ and $e_m \xra{\delta_2} \gamma (e_k; e_{m_{(\pi f)^{-1}(1)}}, \ldots, e_{m_{(\pi f)^{-1}(k)}})$. $\gamma ([f]; [g_1], \ldots, [g_k])$ is represented by the following path
\begin{align*}
  e_m & \xra{\delta_1} \gamma (e_k; e_{m_1}, \ldots) \\
  & \xra{\gamma (f; g_1, \ldots)} \gamma ((\pi f) e_k; (\pi g_1) e_{m_1}, \ldots) = \gamma (\pi f; \pi g_1, \ldots) \gamma (e_k; e_{m_{(\pi f)^{-1} (1)}}, \ldots) \\
  & \xra{\gamma (\pi f; \pi g_1, \ldots) \delta_2^{-1}} \gamma (\pi f; \pi g_1, \ldots) e_m.
\end{align*}
$\delta_1$ and $\gamma (\pi f; \pi g_1, \ldots) \delta_2^{-1}$ will be omitted from now on.

Let $\pi_1(\ms{C}/ \ms{S})= \{\pi_1 (\ms{C}(k)/ S_k; S_ke_k)\}_{k\geq 0}$ and $\pi: \pi_1(\ms{C}/ \ms{S}) \to \ms{S}$ denote the sequence of boundary homomorphisms $\pi_1 (\ms{C}(k)/S_k, S_ke_k) \xra{\pi} S_k$, $k\geq 0$.

\begin{thm}
  If $\ms{C}$ is a path-connected covering symmetric operad with a good basepoint, then $\pi_1 (\ms{C}/\ms{S})$ is a crossed group operad with $\gamma$ and $\pi: \pi_1 (\ms{C}/\ms{S}) \to \ms{S}$ given above.
\end{thm}

\begin{proof}
  Identify $\pi_1 (\ms{C}(k)/S_k, S_ka)$ with $\pi_1 (\ms{C}(k); a, S_ka)$. Unitality is obvious. Since $\gamma ([f]; [g_1], \ldots, [g_k])$ is represented by a path from $e_m$ to $\gamma (\pi[f]; \pi[g_1], \ldots, \pi[g_k]) e_m$, we have
  $$\pi \gamma ([f]; [g_1], \ldots, [g_k])= \gamma (\pi[f]; \pi[g_1], \ldots, \pi[g_k]);$$
  namely $\gamma$ commutes with $\pi$.

  Next we check that $\gamma$ is a crossed homomorphism.
  $\gamma ([f]*[f']; [g_1]*[g'_1], \ldots, [g_k]*[g'_k])$ is represented by
  \begin{align*}
    e_m & \to \gamma (e_k; e_{m_1}, \ldots) \xra{\gamma (f; g_1, \ldots)} \gamma ((\pi f) e_k; (\pi g_1) e_{m_1}, \ldots) \\
    & \xra{\gamma ((\pi f)f'; (\pi g_1) g'_1, \ldots)} \gamma ((\pi f) (\pi f') e_k; (\pi g_1) (\pi g'_1) e_{m_1}, \ldots) \\
    & = \gamma (\pi f \pi f'; \pi g_1 \pi g'_1, \ldots) \gamma (e_k; e_{m_{(\pi f \pi f')^{-1} (1)}}, \ldots) \\
    & \to \gamma (\pi f \pi f'; \pi g_1 \pi g'_1, \ldots) e_m
  \end{align*}
  which is of the form $\zeta_1 \cdot \gamma (f;g_1, \ldots) \cdot \gamma ((\pi f)f'; (\pi g_1) g'_1, \ldots) \cdot \zeta_2$ where $\zeta_1, \zeta_2$ are two linear paths. $\gamma ([f]; [g_1], \ldots, [g_k])* \gamma ([f']; [g'_{[f]^{-1}(1)}], \ldots, [g'_{[f]^{-1}(k)}])$ is represented by
  \begin{align*}
    e_m & \to \gamma (e_k; e_{m_1}, \ldots) \\
    & \xra{\gamma (f; g_1, \ldots)} \gamma ((\pi f) e_k; (\pi g_1) e_{m_1}, \ldots) = \gamma (\pi f; \pi g_1, \ldots) \gamma (e_k; e_{m_{(\pi f)^{-1} (1)}}, \ldots) \\
    & \xra{\gamma (\pi f; \pi g_1, \ldots) \gamma (f'; g'_{(\pi f)^{-1}(1)}, \ldots)} \gamma (\pi f; \pi g_1, \ldots) \gamma ((\pi f') e_k; (\pi g'_{(\pi f)^{-1}(1)}) e_{m_{(\pi f)^{-1}(1)}}, \ldots) \\
    & = \gamma (\pi f; \pi g_1, \ldots) \gamma (\pi f'; \pi g'_{(\pi f)^{-1}(1)}, \ldots) \gamma (e_k; e_{m_{(\pi f)^{-1}((\pi f')^{-1}(1))}}, \ldots) \\
    & \to \gamma (\pi f; \pi g_1, \ldots) \gamma (\pi f'; \pi g'_{(\pi f)^{-1}(1)}, \ldots) e_m
  \end{align*}
  which is of the form $\zeta'_1 \cdot \gamma (f;g_1, \ldots) \cdot \gamma (\pi f; \pi g_1, \ldots) \gamma (f'; g'_{(\pi f)^{-1}(1)}, \ldots) \cdot \zeta'_2$ where $\zeta'_1, \zeta'_2$ are two linear paths. Note that
  \begin{align*}
    \gamma ((\pi f)f'; (\pi g_1) g'_1, \ldots) & = \gamma (\pi f; \pi g_1, \ldots) \gamma (f'; g'_{(\pi f)^{-1}(1)}, \ldots), \\
    \gamma (\pi f \pi f'; \pi g_1 \pi g'_1, \ldots) & = \gamma (\pi f; \pi g_1, \ldots) \gamma (\pi f'; \pi g'_{(\pi f)^{-1}(1)}, \ldots).
  \end{align*}
  Hence $\gamma ([f]*[f']; [g_1]*[g'_1], \ldots)= \gamma ([f]; [g_1], \ldots)* \gamma ([f']; [g'_{[f]^{-1}(1)}], \ldots)$.

  Associativity. $\gamma (\gamma ([f]; [g_1], \ldots, [g_k]); [h_1], \ldots, [h_m])$ is represented by
  \begin{align*}
    e_n & \to \gamma (e_m; e_{n_1}, \ldots) \to \gamma (\gamma (e_k; e_{m_1}, \ldots); e_{n_1}, \ldots) \\
    & \xra{\gamma (\gamma (f;g_1, \ldots); h_1, \ldots)} \gamma (\gamma ((\pi f) e_k; (\pi g_1) e_{m_1}, \ldots); (\pi h_1) e_{n_1}, \ldots) \\
    &= \gamma (\gamma (\pi f; \pi g_1, \ldots) \gamma (e_k;  e_{m_{(\pi f)^{-1} (1)}}, \ldots); (\pi h_1) e_{n_1}, \ldots) \\
    & \to \gamma ((\gamma (\pi f; \pi g_1, \ldots)) e_m; (\pi h_1) e_{n_1}, \ldots) \\
    & = \gamma (\gamma (\pi f; \pi g_1, \ldots); \pi h_1, \ldots) \gamma (e_m; e_{n_{\gamma (\pi f; \pi g_1, \ldots)^{-1}(1)}}, \ldots) \\
    & \to \gamma (\gamma (\pi f; \pi g_1, \ldots); \pi h_1, \ldots) e_n
  \end{align*}
  which is of the form $\zeta_1 \cdot \gamma (\gamma (f;g_1, \ldots); h_1, \ldots) \cdot \zeta_2$ where $\zeta_1, \zeta_2$ are two linear paths. $\gamma ([f]; \gamma ([g_1]; [h_1], \ldots), \ldots, \gamma ([g_k]; \ldots, [h_m]))$ is represented by
  \begin{align*}
    e_n & \to \gamma (e_k; e_{N_1}, \ldots) \to \gamma (e_k; \gamma (e_{m_1}; e_{n_1}, \ldots), \ldots) \\
    & \xra{\gamma (f; \gamma (g_1; h_1, \ldots), \ldots)} \gamma ((\pi f) e_k; \gamma ((\pi g_1) e_{m_1}; (\pi h_1) e_{n_1}, \ldots), \ldots) \\
    & = \gamma ((\pi f) e_k; \gamma (\pi g_1; \pi h_1, \ldots) \gamma (e_{m_1}; e_{n_{(\pi g_1)^{-1} (1)}}, \ldots), \ldots) \\
    & = \gamma ((\pi f) e_k; \gamma (\pi g_1; \pi h_1, \ldots) e_{N_1}, \ldots) \\
    & = \gamma (\pi f; \gamma (\pi g_1; \pi h_1, \ldots), \ldots) \gamma (e_k; e_{N_{(\pi f)^{-1} (1)}}, \ldots) \\
    & \to \gamma (\pi f; \gamma (\pi g_1; \pi h_1, \ldots), \ldots) e_n
  \end{align*}
  which is of the form $\zeta'_1 \cdot \gamma (f; \gamma (g_1; h_1, \ldots), \ldots) \cdot \zeta'_2$ where $\zeta'_1, \zeta'_2$ are two linear paths. Note that
  \begin{align*}
    \gamma (\gamma (f;g_1, \ldots); h_1, \ldots) & = \gamma (f; \gamma (g_1; h_1, \ldots), \ldots), \\
    \gamma (\gamma (\pi f; \pi g_1, \ldots); \pi h_1, \ldots) & = \gamma (\pi f; \gamma (\pi g_1; \pi h_1, \ldots), \ldots).
  \end{align*}
  Hence the associativity holds.
\end{proof}

\subsection{Higher Homotopy Groups of Nonsymmetric Operads}
For convenience, we here make the following conventions: use $+$ to denote the product on $\pi_l$ for $l\geq 1$; $I\vee' I:= (I,1) \vee (I,0)$, $I\vee' S^l := (I,1)\vee (S^l, \textrm{southern pole})$, $S^l\vee' I := (S^l, \textrm{northern pole}) \vee (I,0)$, $S^l\vee' S^l := (S^l, \textrm{northern pole}) \vee (S^l, \textrm{southern pole})$; the wedge $\ms{C}(k) \vee \ms{C}(k')$ is modified accordingly which will not be indicated.

Let $\ms{C}$ be nonsymmetric and $l\geq 1$. Define
$$\gamma: \pi_l (\ms{C}(k), e_k) \times \prod_{i=1}^k \pi_l (\ms{C}(m_i), e_{m_i}) \xra{\gamma_*} \pi_l (\ms{C}(m), \gamma (e_k; e_{m_1}, \ldots)) \xra{\delta_*} \pi_l (\ms{C}(m), e_m)$$
where $\delta_*$ is the isomorphism induced by $e_m \xra{\delta} \gamma (e_k; e_{m_1}, \ldots)$. Namely, for $[f]\in \pi_l (\ms{C}(k), e_k)$ and $[f_i]\in \pi_l (\ms{C}(m_i), e_{m_i})$, $\gamma ([f]; [g_1], \ldots, [g_k])$ is represented by
\begin{align*}
  S^l & \to S^l \vee' S^l \to I\vee' S^l \to I\vee' (S^l)^{1+k} \\
  & \xra{\delta \vee (f\times g_1 \times \cdots \times g_k)} \ms{C}(m) \vee (\ms{C}(k) \times \ms{C}(m_1) \times \cdots \times \ms{C}(m_k)) \\
  & \xra{\id \vee \gamma} \ms{C}(m) \vee \ms{C}(m) \to \ms{C}(m),
\end{align*}
where $S^l \to S^l \vee' S^l$ is the usual projection collapsing the equator to the northern pole and the southern pole of the two $S^l$'s  of $S^l\vee' S^l$ respectively, $S^l \to I$ is the latitude projection with the southern pole mapped to 0 and the northern pole mapped to 1, and
$$\ms{C}(m) \vee \ms{C}(m)= (\ms{C}(m), \gamma (e_k; e_{m_1}, \ldots)) \vee (\ms{C}(m), \gamma (e_k; e_{m_1}, \ldots)) \to \ms{C}(m)$$
is the folding map. If $\{e_k\}_{k\geq 0}$ is a strick basepoint, then $\delta$ is the constant path at $e_m$ and the above composite is homotopic to
$$S^l \to (S^l)^{1+k} \to \ms{C}(k) \times \ms{C}(m_1) \times \cdots \times \ms{C}(m_k) \xra{f\times g_1\times \cdots \times g_k} \ms{C}(m).$$
Next recall that the action of $\pi_1 (\ms{C}(k), e_k)$ on $\pi_l (\ms{C}(k), e_k)$ is given as follows. For $[a]\in \pi_1 (\ms{C}(k), e_k)$, $[a]\star [f]\in \pi_l (\ms{C}(k), e_k)$ is represented by
$$a \star f: S^l \to S^l \vee' S^l \to I\vee' S^l \xra{a\vee f} \ms{C}(k) \vee \ms{C}(k) \to \ms{C}(k).$$
Note that if $l=1$, this action is the conjugation action of $\pi_1 (\ms{C}(k), e_k)$ on itself.

\begin{thm}
  If $\ms{C}$ is a path-connected nonsymmetric operad with a good basepoint, then $\pi_l \ms{C}$ ($l\geq 1$) is a $\pi_1 \ms{C}$-operad.
\end{thm}

\begin{proof}
  Let $[f], [f']\in \pi_l (\ms{C}(k), e_k)$, $[g_i], [g'_i]\in \pi_l (\ms{C}(m_i), e_{m_i})$, $[h_j]\in \pi_l (\ms{C}(n_j), e_{n_j})$, $[a]\in \pi_1 (\ms{C}(k), e_k)$, $[b_i]\in (\ms{C}(m_i), e_{m_i})$.

  Unitality is obvious.

  Associativity. $\gamma (\gamma ([f]; [g_1], \ldots, [g_k]); [h_1], \ldots, [h_m])$ is represented by
  \begin{align*}
    S^l & \to S^l \vee' S^l \to I\vee' S^l \to I\vee' (S^l \times (S^l)^m) \to I\vee' ((S^l\vee' S^l) \times (S^l)^m) \\
    & \to I\vee' ((I\vee' S^l) \times (S^l)^m) \to I\vee' ((I\vee' (S^l)^{1+k}) \times (S^l)^m) \\
    & \to \ms{C}(n) \vee ((\ms{C}(m)\vee (\ms{C}(k)\times \ms{C}(m_1)\times \cdots \times \ms{C}(m_k))) \times \ms{C}(n_1) \times \cdots \times \ms{C}(n_m)) \\
    & \to \ms{C}(n) \vee ((\ms{C}(m) \vee \ms{C}(m)) \times \ms{C}(n_1) \times \cdots \times \ms{C}(n_m)) \\
    & \to \ms{C}(n) \vee (\ms{C}(m)\times \ms{C}(n_1)\times \cdots \times \ms{C}(n_m)) \to \ms{C}(n) \vee \ms{C}(n) \to \ms{C}(n),
  \end{align*}
  which is homotopic to
  \begin{align*}
    S^l & \to S^l \vee' S^l \vee' S^l \to I\vee' I\vee' S^l \to I\vee' I^{1+m} \vee' (S^l)^{1+k+m} \\
    & \to \ms{C}(n) \vee (\ms{C}(m) \times \ms{C}(n_1) \times \cdots) \vee ((\ms{C}(k)\times \ms{C}(m_1) \times \cdots) \times \ms{C}(n_1) \times \cdots) \\
    & \to \ms{C}(n) \vee \ms{C}(n) \vee (\ms{C}(m)\times \ms{C}(n_1) \times \cdots) \\
    & \to \ms{C}(n) \vee \ms{C}(n) \vee \ms{C}(n) \to \ms{C}(n) \vee \ms{C}(n)\vee \ms{C}(n) \to \ms{C}(n).
  \end{align*}
  $\gamma ([f]; \gamma ([g_1]; [h_1], \ldots), \ldots, \gamma ([g_k]; \ldots, [h_m]))$ is represented by
  \begin{align*}
    S^l & \to S^l\vee' S^l \to I\vee' S^l \to I\vee' (S^l\times (S^l)^k) \to I\vee' (S^l\times (S^l\vee' S^l)^k) \\
    & \to I\vee' (S^l\times (I\vee' S^l)^k) \to I\vee' (S^l\times (I\vee' (S^l)^{1+m_1}) \times \cdots) \\
    & \to \ms{C}(n) \vee (\ms{C}(k) \times (\ms{C}(n_1+ \cdots+ n_{m_1}) \vee (\ms{C}(m_1) \times \ms{C}(n_1) \times \cdots)) \times \cdots) \\
    & \to \ms{C}(n) \vee (\ms{C}(k) \times (\ms{C}(n_1+ \cdots+ n_{m_1}) \vee \ms{C}(n_1+ \cdots+ n_{m_1})) \times \cdots) \\
    & \to \ms{C}(n) \vee (\ms{C}(k) \times \ms{C}(n_{n_1+ \cdots+ n_{m_1}}) \times \cdots) \to \ms{C}(n) \vee \ms{C}(n) \to \ms{C}(n),
  \end{align*}
  which is homotopic to
  \begin{align*}
    S^l & \to S^l \vee' S^l \vee' S^l \to I\vee' I\vee' S^l \to I\vee' I^{1+k} \vee' (S^l)^{1+k+m} \\
    & \to \ms{C}(n) \vee (\ms{C}(k) \times \ms{C}(n_1+ \cdots+ n_{m_1}) \times \cdots) \vee (\ms{C}(k)\times (\ms{C}(m_1) \times \ms{C}(n_1) \times \cdots ) \times \cdots) \\
    & \to \ms{C}(n) \vee \ms{C}(n) \vee (\ms{C}(k)\times \ms{C}(n_1+ \cdots+ n_{m_1}) \times \cdots) \\
    & \to \ms{C}(n) \vee \ms{C}(n) \vee \ms{C}(n) \to \ms{C}(n) \vee \ms{C}(n)\vee \ms{C}(n) \to \ms{C}(n).
  \end{align*}
  The associativity hence holds by the one on the space level.

  Equivariance. $\gamma ([a]\star [f]; [b_1]\star [g_1], \ldots, [b_k]\star [g_k])$ is represented by
  \begin{align*}
    S^l & \to S^l \vee' S^l \to I\vee' S^l \to I\vee' (S^l)^{1+k} \to I\vee' (S^l\vee' S^l)^{1+k} \to I\vee' (I\vee' S^l)^{1+k} \\
    & \to \ms{C}(m)\vee ((\ms{C}(k)\vee \ms{C}(k)) \times (\ms{C}(m_1) \vee \ms{C}(m_1)) \times \cdots) \\
    & \to \ms{C}(m) \vee (\ms{C}(k)\times \ms{C}(m_1)\times \cdots) \to \ms{C}(m)\vee \ms{C}(m) \to \ms{C}(m).
  \end{align*}
  Note that $\gamma ([a]; [b_1], \ldots, [b_k])$ is represented by
  \begin{align*}
    c: I &= [0,\frac{1}{3}] \cup [\frac{1}{3}, \frac{2}{3}] \cup [\frac{2}{3},1] \to I\vee' I\vee' I \to I\vee' (I^{1+k})\vee' I \\
    &\xra{\delta \vee (a\times b_1\times \cdots) \vee \delta^{-1}} \ms{C}(m) \vee (\ms{C}(k) \times \ms{C}(m_1) \times \cdots) \vee \ms{C}(m) \\
    &\to \ms{C}(m) \vee \ms{C}(m) \vee \ms{C}(m) \to \ms{C}(m).
  \end{align*}
  Thus $\gamma ([a]; [b_1], \ldots, [b_k]) \star \gamma ([f]; [g_1], \ldots, [g_k])$ is represented by
  \begin{align*}
    S^l &\to S^l\vee' S^l \to I\vee' S^l \to I\vee' (S^l\vee' S^l) \to I\vee' I\vee' (S^l)^{1+k} \\
    &\xra{c\vee \delta\vee (f\times g_1\times \cdots)} \ms{C}(m) \vee \ms{C}(m)\vee (\ms{C}(k)\times \ms{C}(m_1)\times \cdots) \\
    &\to \ms{C}(m)\vee \ms{C}(m)\vee \ms{C}(m) \to \ms{C}(m),
  \end{align*}
  which is homotopic to
  \begin{align*}
    S^l &\to S^l\vee' S^l \to I\vee' S^l \to (I\vee' I)\vee' S^l \to I\vee' (I^{1+k}) \vee' (S^l)^{1+k} \\
    &\to \ms{C}(m) \vee (\ms{C}(k)\times \ms{C}(m_1) \times \cdots) \vee (\ms{C}(k)\times \ms{C}(m_1) \times \cdots) \\
    &\to \ms{C}(m) \vee \ms{C}(m) \vee \ms{C}(m) \to \ms{C}(m),
  \end{align*}
  by noting that $\delta^{-1}\vee \delta: I\vee' I \to \ms{C}(m) \vee \ms{C}(m)$ is null homotopic. Then it is easy to check that the two representing maps are homotopic. The equivariance hence holds.
\end{proof}

It should be noted that $\gamma$ is a homomorphism if $l=1$ but may not be a homomorphism if $l>1$.

\subsection{Higher Homotopy Groups of Symmetric Operads}
Let $\ms{C}$ be a covering symmetric operad and $l>1$. From the short exact sequence
$$1=\pi_l S_k \to \pi_l \ms{C}(k) \to \pi_l (\ms{C}(k)/ S_k) \to \pi_{l-1} S_k=1,$$
$\pi_l (\ms{C}(k)/ S_k)$ can be identified with $\pi_l \ms{C}(k)$. Thus $\pi_l (\ms{C}/ \ms{S})$ inherits the $\gamma$ of $\pi_l \ms{C}$. The action of $\pi_1 (\ms{C}(k)/S_k, S_ke_k)$ on $\pi_l (\ms{C}(k)/ S_k, S_ke_k)= \pi_l (\ms{C}(k),e_k)$ is described as follows. For $[a]\in \pi_1 (\ms{C}(k)/S_k, S_ke_k)$ and $[f]\in \pi_l (\ms{C}(k), e_k)$, $[a]\star [f]\in \pi_l (\ms{C}(k), e_k)$ is represented by
$$a \star f: S^l \to S^l \vee' S^l \to I\vee' S^l \xra{a\vee ((\pi a)f)} \ms{C}(k) \vee \ms{C}(k) \to \ms{C}(k),$$
noting that $a$ is a path from $e_k$ to $(\pi a) e_k$ and $(\pi a)f$ maps the southern pole of $S^l$ to $(\pi a) e_k$.

\begin{thm}
  If $\ms{C}$ is a path-connected covering symmetric operad with a good basepoint, then $\pi_l (\ms{C}/ \ms{S})$ ($l> 1$) is a $\pi_1 (\ms{C}/ \ms{S})$-operad.
\end{thm}

\begin{proof}
  The unitality and associativity of $\pi_l (\ms{C}/ \ms{S})$ follow from those of $\pi_l \ms{C}$. The equivariance is checked as follows. $\gamma ([a]\star [f]; [b_1]\star [g_1], \ldots, [b_k]\star [g_k])$ is represented by
  \begin{align*}
    S^l & \to S^l \vee' S^l \to I\vee' S^l \to I\vee' (S^l)^{1+k} \to I\vee' (S^l\vee' S^l)^{1+k} \to I\vee' (I\vee' S^l)^{1+k} \\
    & \xra{\delta \vee ((a\vee (\pi a)f) \times (b_1\vee (\pi b_1)g_1) \times \cdots)} \ms{C}(m)\vee ((\ms{C}(k)\vee \ms{C}(k)) \times (\ms{C}(m_1) \vee \ms{C}(m_1)) \times \cdots) \\
    & \to \ms{C}(m) \vee (\ms{C}(k)\times \ms{C}(m_1)\times \cdots) \to \ms{C}(m)\vee \ms{C}(m) \to \ms{C}(m).
  \end{align*}
  Note that $\gamma ([a]; [b_1], \ldots, [b_k])$ is represented by
  \begin{align*}
    c: I &= [0,\frac{1}{3}] \cup [\frac{1}{3}, \frac{2}{3}] \cup [\frac{2}{3},1] \to I\vee' I\vee' I \to I\vee' (I^{1+k})\vee' I \\
    &\xra{\delta \vee (a\times b_1\times \cdots) \vee \gamma (\pi a; \pi b_1, \ldots) \zeta^{-1}} \ms{C}(m) \vee (\ms{C}(k) \times \ms{C}(m_1) \times \cdots) \vee \ms{C}(m) \\
    &\to \ms{C}(m) \vee \ms{C}(m) \vee \ms{C}(m) \to \ms{C}(m).
  \end{align*}
  Thus $\gamma ([a]; [b_1], \ldots, [b_k]) \star \gamma ([f]; [g_1], \ldots, [g_k])$ is represented by
  \begin{align*}
    S^l &\to S^l\vee' S^l \to I\vee' S^l \to I\vee' (S^l\vee' S^l) \to I\vee' I\vee' (S^l)^{1+k} \\
    &\xra{c\vee \gamma (\pi a; \pi b_1, \ldots) \zeta \vee (f\times g_{a^{-1}(1)} \times \cdots)} \ms{C}(m) \vee \ms{C}(m)\vee (\ms{C}(k)\times \ms{C}(m_1)\times \cdots) \\
    &\to \ms{C}(m)\vee \ms{C}(m)\vee \ms{C}(m) \xra{\id \vee \id \vee \gamma (\pi a; \pi b_1, \ldots)} \ms{C}(m)\vee \ms{C}(m)\vee \ms{C}(m) \to \ms{C}(m),
  \end{align*}
  which is homotopic to
  \begin{align*}
    S^l &\to S^l\vee' S^l \to I\vee' S^l \to (I\vee' I)\vee' S^l \to I\vee' (I^{1+k}) \vee' (S^l)^{1+k} \\
    &\xra{\delta \vee (a\times b_1 \times \cdots) \vee (f\times g_{a^{-1}(1)} \times \cdots)} \ms{C}(m) \vee (\ms{C}(k) \times \ms{C}(m_1) \times \cdots) \vee (\ms{C}(k) \times \ms{C}(m_1) \times \cdots) \\
    &\to \ms{C}(m) \vee \ms{C}(m) \vee \ms{C}(m) \xra{\id \vee \id \vee \gamma (\pi a; \pi b_1, \ldots)} \ms{C}(m)\vee \ms{C}(m)\vee \ms{C}(m) \to \ms{C}(m),
  \end{align*}
  by noting that $\gamma (\pi a; \pi b_1, \ldots) \zeta^{-1} \vee \gamma (\pi a; \pi b_1, \ldots) \zeta: I\vee' I \to \ms{C}(m) \vee \ms{C}(m)$ is null homotopic and by the equivariance of $\gamma$ on the space level. Then it is easy to check that the two representing maps are homotopic. The equivariance hence holds.
\end{proof}

\begin{rem}
  If $l=1$, $a\star b= aba^{-1}$ for $a,b\in \pi_1 (\ms{C}(k)/S_k, S_k e_k)$, i.e., $\star$ is the conjugation action. Observe that
  \begin{align*}
    \gamma (a\star a'; b_1\star b'_1, \ldots) &= \gamma (aa'a^{-1}; b_1b'_1b_1^{-1}, \ldots) \\
    &= \gamma (a;b_1, \ldots) \gamma (a';b'_1, \ldots) \gamma (a; b_{(aa'a^{-1})^{-1}(1)}, \ldots)^{-1}, \\
    \gamma (a; b_1, \ldots) \star \gamma (a'; b'_1, \ldots) &= \gamma (a; b_1, \ldots) \gamma (a'; b'_1, \ldots) \gamma (a; b_1, \ldots)^{-1}.
  \end{align*}
  Note that in general $(aa'a^{-1})^{-1}(i) \neq i$, thus
  $$\gamma (a\star a'; b_1\star b'_1, \ldots) \neq \gamma (a; b_1, \ldots) \star \gamma (a'; b'_1, \ldots).$$
  So the equivariance does not hold. Namely, $\pi_1 (\ms{C}/ \ms{S})$ does not act on itself by the conjugation action.
\end{rem}

\subsection{Homotopy Groups of $\ms{G}$-Operads}
A minor modification of the above discussion about the homotopy groups of symmetric operads gives the following result for general $\ms{G}$-operads.

Let $\ms{G}$ be a group operad, $\ms{H}$ a non-crossed normal sub group operad of $\ms{G}$ and $\ms{C}$ a $\ms{G}$-operad. Recall that $\ms{C}/ \ms{H}$ is $\ms{G}/ \ms{H}$-operad.

\begin{thm}\label{thm:homotopy_group_G-operads}
  If $\ms{C}/ \ms{H}$ is a path-connected covering $\ms{G}/ \ms{H}$-operad with a good basepoint, then $\pi_1 (\ms{C}/ \ms{H})$ is a non-crossed group operad, $\pi_1 (\ms{C}/ \ms{G})= \pi_1 ((\ms{C}/ \ms{H})/ (\ms{G}/ \ms{H}))$ is a group operad with a short exact sequence of group operads
  $$1\to \pi_1 (\ms{C}/ \ms{H})\to \pi_1 (\ms{C}/ \ms{G})\to \ms{G}/ \ms{H}\to 1,$$
  and $\pi_l (\ms{C}/ \ms{G})$ is a $\pi_1 (\ms{C}/ \ms{G})$-operad for $l\geq 1$ if $\ms{G}$ non-crossed and for $l>1$ if $\ms{G}$ crossed. \qed
\end{thm}

Thus $\pi_1 (\ms{C}/\ms{G})$ is non-crossed if $\ms{G}$ is non-crossed, and crossed if $\ms{G}$ is crossed.

$\pi_l (\ms{C}/ \ms{G})$ ($l>1$) should be a simplicial group though $\gamma$ is not multilinear nor linear. It should be possible to generalize this theorem to $\la A, \ms{C}\ra$ where $A$ is a connected CW complex with a vertex as the basepoint.

\begin{rem}
  If $\ms{G}$ non-crossed, without the requirement $\ms{C}/\ms{H}$ a covering $\ms{G}/ \ms{H}$-operad, $\pi_1 (\ms{C}/ \ms{G})$ is still a group operad since $\ms{C}/ \ms{G}$ is a nonsymmetric operad. However, there may not be the short exact sequence $1\to \pi_1 (\ms{C}/ \ms{H})\to \pi_1 (\ms{C}/ \ms{G})\to \ms{G}/ \ms{H}\to 1$, since $G_k/H_k\to \ms{C}(k)/H_k\to \ms{C}(k)/G_k$ may not be a fibration. If $\ms{G}$ crossed, without the requirement $\ms{C}/\ms{H}$ a covering $\ms{G}/ \ms{H}$-operad, $\pi_1 (\ms{C}/ \ms{G})$ may not be a group operad due to the ``crossed homomorphism'' condition.
\end{rem}

We make the convention that the homotopy groups operads of a topological $\ms{G}$-operad $\ms{C}$ refer to $\pi_l (\ms{C}/\ms{G})$, $l\geq 1$. In particular, the fundamental groups operad of a topological $\ms{G}$-operad $\ms{C}$ refers to $\pi_1 (\ms{C}/\ms{G})$.

$\pi_l (\ms{C}/\ms{G})$ ($l\geq 1$) may be regarded as functors. We next consider their naturality. Let $\ms{C}$, $\ms{C}'$ be two $\ms{G}$-operads such that both $\ms{C}/ \ms{H}$, $\ms{C}'/ \ms{H}$ are path-connected covering $\ms{G}/ \ms{H}$-operads with good basepoints.

\begin{prop}
  A pointed morphism $\psi: \ms{C}\to \ms{C}'$ of $\ms{G}$-operads induces morphisms $\psi_*: \pi_1 (\ms{C}/ \ms{H})\to \pi_1 (\ms{C}'/ \ms{H})$ and $\psi_*: \pi_1 (\ms{C}/ \ms{G})\to \pi_1 (\ms{C}'/ \ms{G})$ of group operads such that the following diagram is commutative
  $$\xymatrix{
    1 \ar[r] & \pi_1 (\ms{C}/\ms{H}) \ar[d]_-{\psi_*} \ar[r] & \pi_1 (\ms{C}/\ms{G}) \ar[d]_-{\psi_*} \ar[r] & \ms{G}/\ms{H} \ar@{=}[d] \ar[r] & 1 \\
    1 \ar[r] & \pi_1 (\ms{C}'/\ms{H}) \ar[r] & \pi_1 (\ms{C}'/\ms{G}) \ar[r] & \ms{G}/\ms{H} \ar[r] & 1 }$$
  and morphisms $\psi_*: \pi_l (\ms{C}/ \ms{S})\to \pi_l (\ms{C}'/ \ms{S})$ of $\pi_1 (\ms{C}/ \ms{S})$-operads. If $\psi: \ms{C}/\ms{H} \to \ms{C}'/\ms{H}$ is an equivalence, then these induced morphisms are isomorphisms. Conversely, if all $\psi_*: \pi_l (\ms{C}/ \ms{H})\to \pi_l (\ms{C}'/ \ms{H})$, $l\geq 1$, are isomorphisms, then $\psi: \ms{C}/\ms{H} \to \ms{C}'/\ms{H}$ is an equivalence.
\end{prop}

\begin{proof}
  Commutativity of the diagram follows from the following morphism of fibrations
  $$\xymatrix{
    \ms{G}/\ms{H} \ar@{=}[d] \ar[r] & \ms{C}/\ms{H} \ar[d] \ar[r] & \ms{C}/\ms{G} \ar[d] \\
    \ms{G}/\ms{H} \ar[r] & \ms{C}'/\ms{H} \ar[r] & \ms{C}'/\ms{G} }$$
  If all $\psi_*: \pi_1 (\ms{C}/ \ms{H})\to \pi_1 (\ms{C}'/ \ms{H})$, $l\geq 1$, are isomorphisms, all $\psi: \ms{C}(k)/H_k\to \ms{C}'(k)/H_k$ are homotopy equivalences. Then $\psi: \ms{C}/\ms{H} \to \ms{C}'/\ms{H}$ is an equivalence since both $\ms{C}/\ms{H}$, $\ms{C}'/\ms{H}$ are covering $\ms{G}/\ms{H}$-operads. The rest can be easily verified.
\end{proof}

Recall from Corollary 4.5 of \cite{May:1972:GILS}, for $n\geq 3$,
$$\pi_l \ms{C}_n(k)= \bigoplus_{i=1}^{k-1} \pi_l \bigvee^i S^{n-1}.$$
Thus for $n\geq 3$,
$$\left\{\bigoplus_{i=1}^{k-1} \pi_l \bigvee^i S^{n-1}\right\}_{k\geq 0}$$
admits a natural symmetric operad structure. It may be interesting to investigate this symmetric operad structure.

\section{Covering Operads}
In this section, we study covering operads which is an analogue of covering spaces. It should be possible to develop for covering operads an analogous theory of covering spaces. However we shall only consider universal cover of $\ms{G}$-operads and a canonical construction of a universal $\ms{G}$-operad $E\ms{G}$ for a group operad $\ms{G}$, which are analogues of the universal cover of spaces and the canonical construction $EG\to BG$ for a group $G$, respectively. We then show that any group operad $\ms{G}$ can be realized as the fundamental groups operad of $E\ms{G}$. In the last subsection, we shall apply this theory to give an algebraic characterization of $K(\pi,1)$ operads by their fundamental groups operads, and then reconstruct them from their fundamental groups operads.

Let $\ms{G}' \onto \ms{G}$ be an epimorphism of group operads, $\ms{C}'$ a $\ms{G}'$-operad and $\ms{C}$ a $\ms{G}$-operad. Regard $\ms{C}$ as a $\ms{G}'$-operad via $\ms{G}'\onto \ms{G}$.

\begin{defn}
  A morphism $\psi: \ms{C}' \to \ms{C}$ of $\ms{G}'$-operads is called a \textbf{covering morphism} and $\ms{C}'$ together with $\psi$ is called a \textbf{covering operad} of $\ms{C}$, if each $\psi: \ms{C}'(k)\to \ms{C}(k)$ is a covering map. A covering operad $(\ms{C}', \psi)$ of $\ms{C}$ is called a \textbf{universal cover} of $\ms{C}$, if each $\psi: \ms{C}'(k)\to \ms{C}(k)$ is a universal cover.
\end{defn}

Quotients of operads provide examples of covering operads. If $\ms{H}$ is a non-crossed normal sub group operad of $\ms{G}$ and its action on $\ms{C}$ is a covering action, then the quotient morphism $\ms{C}\to \ms{C}/ \ms{H}$ is a covering morphism and $\ms{C}$ with this quotient morphism is a covering operad of $\ms{C}/ \ms{H}$.

Recall that $\ms{G}$-operads can be converted into nonsymmetric or symmetric operads by restricting the action of $\ms{G}$ to the trivial group operad or by taking quotients. Conversely, nonsymmetric and symmetric operads may also be converted into $\ms{G}$-operads by taking their covering operads.

\subsection{Universal Cover of $\ms{G}$-Operads}
We shall give a construction of universal cover of $\ms{G}$-operads in the following. Part of the idea of the construction is motivated by Fiedorowicz's construction \cite{Fiedorowicz:preprint:SBC} of a universal cover of the little 2-cubes operad $\ms{C}_2$.

If $\ms{C}$ is a path-connected covering topological $\ms{G}$-operad with a good basepoint, by definition, a universal cover of $\ms{C}$ is a topological $\pi_1 (\ms{C}/ \ms{G})$-operad $\wt{\ms{C}}$ together with a morphism $\wt{\ms{C}}\to \ms{C}$ of $\pi_1 (\ms{C}/ \ms{G})$-operads such that each $\wt{\ms{C}}(k) \to \ms{C}(k)$ is a universal cover, where $\ms{C}$ is regarded as a $\pi_1 (\ms{C}/ \ms{G})$-operad via $\pi_1 (\ms{C}/ \ms{G})\to \ms{G}$.

To equip a natural action of group operad on the universal cover of $\ms{G}$-operads, let us first discuss group action on the universal cover of $G$-spaces.

Suppose $(X,e)$ is path-connected, locally path-connected and semilocally simply-connected and admits a covering action of a discrete group $G$. Then $(X,e)$ has a universal cover $(\wt{X}, \wt{e})$ (cf. \cite{Hatcher:2002:AT}, page 64) which can be chosen as
$$\wt{X}= \{[\alpha] \mid \alpha \textrm{ is a path in } X \textrm{ starting at } e\},$$
and $\wt{e}=[e]$ the homotopy class of the constant path at $e$. There is a natural action of $\pi_1 (X/G)$ on $\wt{X}$ described as follows. Since $G\to X\to X/G$ is a fibration, there is the short exact sequence
$$1\to \pi_1X \to \pi_1 (X/G) \xra{\pi} G\to 1.$$
Identify $\pi_1 (X/G)$ with $\pi_1 (X; e, Ge)$ and also denote $\pi: \pi_1 (X; e, Ge)= \pi_1 (X/G) \xra{\pi} G$. Each homotopy class of loops in $\pi_1 (X/G)$ can be represented by certain path $[f]\in \pi_1 (X; e, Ge)$ from $e$ to $(\pi f)e$ where $\pi f= \pi [f]\in G$. $\pi_1 (X/G)$ acts on $\wt{X}$ by
$$[f] \cdot [\alpha]:= [f\cdot (\pi f) \alpha]= [e\xra{f} (\pi f)e \xra{(\pi f) \alpha} (\pi f) \alpha(1)]$$
where $[f]\in \pi_1(X/G)= \pi_1 (X; e, Ge)$, $[\alpha]\in \wt{X}$ and $(\pi f) \alpha$ is the translation of $\alpha$ by $\pi f\in G$ which is a path from $(\pi f)e$ to $(\pi f) \alpha(1)$. The action of $\pi_1 (X/G)$ on $\wt{X}$ is a covering action. With this action, the projection $p: \wt{X} \to X$, $[\alpha] \mapsto \alpha(1)$, is a $\pi_1 (X/G)$-equivariant map, where $\pi_1 (X/G)$ acts on $X$ via $\pi: \pi_1 (X/G) \to G$. Moreover, $p: \wt{X} \to X$ factors through $\wt{X}/ \pi_1X$ such that $\wt{X}/ \pi_1X \to X$ is a $G$-homeomorphism, i.e.,
$$\xymatrix{
  \wt{X} \ar[d] \ar[r]^p & X  \\
  \wt{X}/ \pi_1 X \ar@{-->}[ur]_{\cong}}$$

$\wt{X}$ and the action of $\pi_1 (X/G)$ on $\wt{X}$ are natural with respect to $G$-maps as follows. Suppose $(X',e')$ as well is path-connected, has a universal cover $(\wt{X}', \wt{e}')$ and admits a covering action of $G$, and $\phi: X\to X'$ is a pointed $G$-map. $\phi$ gives rise to a morphism of fibrations
$$\xymatrix{
  G \ar@{=}[d] \ar[r] & X \ar[d]_{\phi} \ar[r] & X/G \ar[d]^{\phi} \\
  G \ar[r] & X' \ar[r] & X'/G   }$$
inducing the following commutative diagram of short exact sequences
$$\xymatrix{
  1 \ar[r] & \pi_1 X \ar[d]_{\phi_*} \ar[r] & \pi_1 (X/G) \ar[d]_{\phi_*} \ar[r]^-{\pi} & G \ar@{=}[d] \ar[r] & 1 \\
  1 \ar[r] & \pi_1 X' \ar[r] & \pi_1 (X'/G) \ar[r]^-{\pi} & G \ar[r] &1  }$$
Thus $\pi f= \pi (\phi_*[f])= \pi (\phi\circ f)$. Then
$$\wt{\phi}: \wt{X}\to \wt{X}', \quad [\alpha] \mapsto [\phi \circ \alpha]$$
is equivariant in the sense that $\wt{\phi} ([f][\alpha])= (\phi_*[f]) \wt{\phi} [\alpha]$, and the following diagram is commutative
$$\xymatrix{
  \wt{X} \ar[d] \ar[r]^{\wt{\phi}} & \wt{X}' \ar[d] \\
  X \ar[r]^{\phi} & X'   }$$
$\wt{\phi}$ is equivariant since
\begin{align*}
  \wt{\phi} ([f][\alpha]) &= \wt{\phi} [f\cdot (\pi f)\alpha]= [\phi \circ (f\cdot (\pi f)\alpha)]= [(\phi \circ f) \cdot (\phi \circ (\pi f) \alpha)] \\
  &= [(\phi \circ f) \cdot (\pi f) (\phi \circ \alpha)]= [(\phi \circ f) \cdot (\pi (\phi\circ f)) (\phi \circ \alpha)] \\
  &= [\phi \circ f] [\phi \circ \alpha]= (\phi_*[f]) (\wt{\phi} [\alpha]).
\end{align*}
Moreover, if $\phi$ is a homotopy equivalence, then so is $\wt{\phi}$; if $\phi$ is an equivariant homotopy equivalence, then so is $\wt{\phi}$.

Let $\ms{C}$ be a path-connected, locally path-connected and semilocally simply-connected covering $\ms{G}$-operad with a good basepoint $\{e_k\}_{k\geq 0}$. Then each $\ms{C}(k)$ has a universal cover $p: (\wt{\ms{C}}(k), [e_k])\to (\ms{C}(k),e_k)$ chosen as above and there is a covering action of $\pi_1 (\ms{C}(k)/G_k)$ on $\wt{\ms{C}}(k)$ such that $p$ is $\pi_1 (\ms{C}(k)/G_k)$-equivariant for each $k$. Define $\wt{\gamma}$ as the unique lifting of $\gamma \circ p$ in the following diagram
$$\xymatrix{
\wt{\ms{C}}(k) \times \wt{\ms{C}}(m_1) \times \cdots \times \wt{\ms{C}}(m_k) \ar[d]_{p} \ar@{-->}[r]^-{\wt{\gamma}} &\wt{\ms{C}}(m) \ar[d]^{p}  \\
\ms{C}(k) \times \ms{C}(m_1) \times \cdots \times \ms{C}(m_k) \ar[r]^-{\gamma} & \ms{C}(m) }$$
such that $\wt{\gamma} ([e_k]; [e_{m_1}], \ldots, [e_{m_k}])= [e_m\to \gamma (e_k; e_{m_1}, \ldots, e_{m_k})]$ the homotopy class of the linear path from $e_m$ to $\gamma (e_k; e_{m_1}, \ldots, e_{m_k})$. Explicitly,
$$\wt{\gamma} ([\alpha]; [\beta_1], \ldots)= [e_m\to \gamma (e_k; e_{m_1}, \ldots) \xra{\gamma (\alpha; \beta_1, \ldots)} \gamma (\alpha(1); \beta_1(1), \ldots)],$$
for $[\alpha]\in \wt{\ms{C}}(k)$ and $[\beta_i]\in \wt{\ms{C}}(m_i)$.

\begin{thm}
  If $\ms{C}$ is a path-connected, locally path-connected and semilocally simply-connected covering $\ms{G}$-operad with a good basepoint, then $\wt{\ms{C}}$ is a universal cover of $\ms{C}$ and is a covering $\pi_1 (\ms{C}/ \ms{G})$-operad. Moreover, the projection $p: \wt{\ms{C}}\to \ms{C}$ factors through $\wt{\ms{C}}/ \pi_1 \ms{C}$ such that $\wt{\ms{C}}/ \pi_1 \ms{C}\to \ms{C}$ is an isomorphism of $\ms{G}$-operads, i.e.,
  $$\xymatrix{
  \wt{\ms{C}} \ar[d] \ar[r]^p & \ms{C}  \\
  \wt{\ms{C}}/ \pi_1 \ms{C} \ar@{-->}[ur]_{\cong}}$$
\end{thm}

\begin{proof}
  It suffices to check that $\wt{\ms{C}}$ is a $\pi_1 (\ms{C}/ \ms{G})$-operad. It is evident that $[e_1]$ is the unit of $\wt{\gamma}$. Associativity is checked next. Note that $\wt{\gamma} (\wt{\gamma} (-);-)$ is the lifting of $\gamma (\gamma (-);-)$ such that
  $$\wt{\gamma} (\wt{\gamma} ([e_k]; [e_{m_1}], \ldots); [e_{n_1}], \ldots)= \wt{\gamma} ([e_m \to \gamma (e_k; e_{m_1}, \ldots)]; [e_{n_1}], \ldots)$$
  which is represented by the linear path
  $$e_n\to \gamma (e_m; e_{n_1}, \ldots) \to \gamma (\gamma (e_k; e_{m_1}, \ldots); e_{n_1}, \ldots),$$
  and $\wt{\gamma} (-; \wt{\gamma} (-), \ldots)$ is the lifting of $\gamma (-; \gamma (-), \ldots)$ such that
  $$\wt{\gamma} ([e_k]; \wt{\gamma} ([e_{m_1}]; [e_{n_1}] , \ldots), \ldots)= \wt{\gamma} ([e_k]; [e_{n_1+ \cdots+ n_{m_1}}\to \gamma (e_{m_1}; e_{n_1} , \ldots)], \ldots)$$
  which is represented by the linear path
  $$e_n\to \gamma (e_k; e_{n_1+ \cdots+ n_{m_1}}, \ldots) \to \gamma (e_k; \gamma (e_{m_1}; e_{n_1}, \ldots), \ldots).$$
  The two linear paths are homotopic relative to the endpoints since
  $$\gamma (\gamma (e_k; e_{m_1}, \ldots); e_{n_1}, \ldots)= \gamma (e_k; \gamma (e_{m_1}; e_{n_1}, \ldots), \ldots).$$
  Hence the associativity of $\wt{\gamma}$ holds. To check the equivariance, it suffices by the unique lifting property to show
  $$\gamma ([f][e_k]; [g_1][e_{m_1}], \ldots)= \gamma ([f]; [g_1], \ldots) \gamma ([e_k]; [e_{m_{(\pi f)^{-1}(1)}}], \ldots)$$
  for $[f]\in \pi_1 (\ms{C}(k)/G_k)$ and $[g_i]\in \pi_1 (\ms{C}(m_i)/G_{m_i})$. $\gamma ([f][e_k]; [g_1][e_{m_1}], \ldots)$ is represented by
  \begin{align*}
    e_m &\to \gamma (e_k; e_{m_1}, \ldots) \\
    &\xra{\gamma (f; g_1, \ldots)} \gamma ((\pi f)e_k; (\pi g_1) e_{m_1}, \ldots)= \gamma (\pi f; \pi g_1, \ldots) \gamma (e_k; e_{m_{(\pi f)^{-1}(1)}}, \ldots).
  \end{align*}
  Note that $\gamma ([f]; [g_1], \ldots)$ is represented by
  \begin{align*}
    e_m &\to \gamma (e_k; e_{m_1}, \ldots) \\
    &\xra{\gamma (f; g_1, \ldots)} \gamma ((\pi f)e_k; (\pi g_1) e_{m_1}, \ldots)= \gamma (\pi f; \pi g_1, \ldots) \gamma (e_k; e_{m_{(\pi f)^{-1}(1)}}, \ldots) \\
    &\to \gamma (\pi f; \pi g_1, \ldots) e_m,
  \end{align*}
  and $\gamma ([e_k]; [e_{m_{(\pi f)^{-1}(1)}}], \ldots)$ is represented by
  $$e_m \to \gamma (e_k; e_{m_{(\pi f)^{-1}(1)}}, \ldots).$$
  Thus $\gamma ([f]; [g_1], \ldots) \gamma ([e_k]; [e_{m_{(\pi f)^{-1}(1)}}], \ldots)$ is represented by
  \begin{align*}
    e_m &\to \gamma (e_k; e_{m_1}, \ldots) \\
    &\xra{\gamma (f; g_1, \ldots)} \gamma ((\pi f)e_k; (\pi g_1) e_{m_1}, \ldots)= \gamma (\pi f; \pi g_1, \ldots) \gamma (e_k; e_{m_{(\pi f)^{-1}(1)}}, \ldots) \\
    &\to \gamma (\pi f; \pi g_1, \ldots) e_m \\
    &\to \gamma (\pi f; \pi g_1, \ldots) \gamma (e_k; e_{m_{(\pi f)^{-1}(1)}}, \ldots),
  \end{align*}
  where the latter two arrows can be canceled obviously. Hence the equivariance holds.
\end{proof}

This construction of universal cover is natural with respect to morphisms of $\ms{G}$-operads in the sense described as follows. Let $\ms{C}$ and $\ms{C}'$ be two $\ms{G}$-operads satisfying the condition in the theorem, and $\psi: \ms{C}\to \ms{C}'$ a morphism of $\ms{G}$-operads. $\wt{\ms{C}'}$ can be regarded as a $\pi_1 (\ms{C}/\ms{G})$-operad via $\psi_*: \pi_1 (\ms{C}/\ms{G})\to \pi_1 (\ms{C}'/\ms{G})$. Then $\wt{\psi}: \wt{\ms{C}} \to \wt{\ms{C}'}$ is a morphism of $\pi_1 (\ms{C}/\ms{G})$-operads and the following diagram is commutative
$$\xymatrix{
  \wt{\ms{C}} \ar[d] \ar[r]^{\wt{\psi}} & \wt{\ms{C}}' \ar[d] \\
  \ms{C} \ar[r]^{\psi} & \ms{C}'   }$$
Consequently, if $\psi$ is an equivalence then so is $\wt{\psi}$; thus if $\ms{C}$ and $\ms{C}'$ are equivalent, then so are $\wt{\ms{C}}$ and $\wt{\ms{C}}'$.

As a special case, if $\ms{C}$ is $K(\pi,1)$, then its universal cover $\wt{\ms{C}}$ is a universal $\pi_1 (\ms{C}/ \ms{G})$-operad. For instance, the little 2-cubes operad $\ms{C}_2$ is a $K(\pi,1)$ symmetric operad and thus its universal cover $\wt{\ms{C}}_2$ is a universal $\pi_1 (\ms{C}/ \ms{G})$-operad, i.e., a universal $\ms{B}$-operad, since $\pi_1 (\ms{C}/ \ms{G}) \cong \ms{B}$. $\wt{\ms{C}}_2$ is first constructed in \cite{Fiedorowicz:preprint:SBC}. There is another example related to the ribbon braid groups operad $\ms{R}$ discussed in Wahl's Ph.D. thesis \cite{Wahl:2001:PhD}.

\subsection{Universal $\ms{G}$-Operads}
We here give a construction of a universal $\ms{G}$-operad $E\ms{G}$ for any group operad $\ms{G}$, using the canonical construction $EG\to BG$ for a group $G$, and then show that the fundamental groups operad of $E\ms{G}$ is exactly $\ms{G}$. Constructions for the cases $\ms{G}= \ms{S}$, $\ms{G}= \ms{B}$ and $\ms{G}= \ms{R}$ are already studied in \cite{Bar-Ecc:1974:I}, \cite{Fiedorowicz:preprint:SBC} and \cite{Wahl:2001:PhD} respectively.

Given a group $G$, let $EG= \{(EG)_k\}_{k\geq 0}$ be the simplicial group with $(EG)_k= G^{k+1}$ and faces $d_i$, degeneracies $s_i$ defined as
$$d_i (a_0,\ldots,a_k)= (a_0,\ldots, \hat{a}_i, \ldots, a_k), \quad s_i (a_0, \ldots, a_k)= (a_0, \ldots, a_i, a_i, \ldots, a_k).$$
The geometric realization $|EG|$ of $EG$ is contractible, and $G$ acts on $EG$ by
$$b \cdot (a_0, \ldots, a_k)= (b a_0, \ldots, b a_k),$$
and thus on $|EG|$ by
$$b \cdot [a_0, \ldots, a_k;t]= [b a_0, \ldots, b a_k;t]$$
for $t\in \Delta^k$. The actions of $G$ on $EG$ and $|EG|$ are free and moreover the one on $|EG|$ is a covering action.

\begin{prop}
  For a group operad $\ms{G}= \{G_n\}_{n\geq 0}$, $E\ms{G}= \{EG_n\}_{n\geq 0}$ is a simplicial $\ms{G}$-operad.
\end{prop}

\begin{proof}
  The composition $\gamma$ of $\ms{G}$ induces a simplicial map
  \begin{align*}
    \gamma= E\gamma & : EG_k \times EG_{m_1} \times \cdots \times EG_{m_k} \to EG_m, \\
    \gamma (c; a_1, \ldots, a_k) & = (\gamma (c_0;a_{10},\ldots,a_{k0}), \ldots, \gamma (c_l;a_{1l},\ldots,a_{kl})),
  \end{align*}
  where $c= (c_0,\ldots,c_l)$, $a_i= (a_{i0},\ldots,a_{il})$. With this map, it is easy to check that $E\ms{G}$ is a simplicial $\ms{G}$-operad.
\end{proof}

\begin{prop}
  $|E\ms{G}|= \{|EG_n|\}_{n\geq 0}$ is a universal $\ms{G}$-operad. \qed
\end{prop}

For any subgroup $H\leq G$, let $EG/H= (EG)/H$ be the quotient of $EG$ modulo the action of $H$ with
$$(EG/H)_k= (EG)_k/H= G^{k+1}/H= \{H(a_0, \ldots, a_k) \mid a_i\in G\}.$$
The projection $|EG|\to |EG/H|$ factors through $|EG|/H \xra{\cong} |EG/H|$ which is indeed a homeomorphism. Thus we shall identify $|EG|/H$ and $|EG/H|$.

\begin{prop}
  If $\ms{H}$ is a non-crossed normal sub group operad of $\ms{G}$, then $E\ms{G}/ \ms{H}$ is a simplicial $\ms{G}/ \ms{H}$-operad and $|E\ms{G}|/ \ms{H}$ is a topological $\ms{G}/ \ms{H}$-operad.
\end{prop}

This assertion is actually valid for more general simplicial $\ms{G}$-operads.

\begin{proof}
  The $\gamma$ of $E\ms{G}$ induces a $\gamma$ for $E\ms{G}/ \ms{H}$,
  $$\gamma (H_kc; H_{m_1}a_1, \ldots, H_{m_k}a_k)= H_m \gamma (c; a_1, \ldots, a_k),$$
  where $c=(c_0,\ldots,c_l)\in (EG_k)_l$, $a_i=(a_{i0},\ldots,a_{il})\in (EG_{m_i})_l$, and there is a commutative diagram of simplicial maps
  \begin{diagram}
    EG_k \times EG_{m_1} \times \cdots \times EG_{m_k}  &\rTo^{\gamma}  &EG_m \\
    \dTo  &  &\dTo \\
    EG_k/H_k \times EG_{m_1}/H_{m_1} \times \cdots \times EG_{m_k}/H_{m_k}  &\rTo^{\gamma}  &EG_m/H_m,
  \end{diagram}
  since $\ms{H}$ is non-crossed and thus for $b\in H_k$, $b_i\in H_{m_i}$,
  \begin{align*}
    &\ \gamma (bc; b_1a_1, \ldots, b_ka_k) \\
    =&\ \gamma (b(c_0,\ldots,c_l); b_1(a_{10},\ldots,a_{1l}), \ldots, b_k(a_{k0},\ldots,a_{kl})) \\
    = &\ (\gamma (bc_0; b_1a_{10}, \ldots, b_ka_{k0}), \ldots, \gamma (bc_l; b_1 a_{1l}, \ldots, b_k a_{kl})) \\
    = &\ (\gamma(b; b_1, \ldots, b_k) \gamma (c_0;a_{10},\ldots,a_{k0}), \ldots, \gamma(b; b_1, \ldots, b_k) \gamma (c_l;a_{1l},\ldots,a_{kl})) \\
    = &\ \gamma(b;b_1,\ldots,b_k) (\gamma (c_0;a_{10},\ldots,a_{k0}), \ldots, \gamma (c_l;a_{1l},\ldots,a_{kl})) \\
    =&\ \gamma (b,b_1,\ldots,b_k) \gamma (c;a_1,\ldots,a_k).
  \end{align*}
  Associativity obviously holds, as well as unitality since
  $$\gamma (H_1e_1^{(l)}; H_n(a_0,\ldots,a_l))= H_n\gamma (e_1^{(l)}; a_0,\ldots,a_l)= H_n(a_0,\ldots,a_l),$$
  $$\gamma (H_n(a_0,\ldots,a_l); H_1e_1^{(l)}, \ldots, H_1e_1^{(l)})= H_n\gamma ((a_0,\ldots,a_l); e_1^{(l)}, \ldots, e_1^{(l)})= H_n(a_0,\ldots,a_l).$$
  $G_n/H_n$ acts on $EG_n/H_n$ by $(H_nb) \cdot H_n(a_0,\ldots,a_k)= H_n(ba_0, \ldots, ba_k)$. This action satisfies the equivariance, since for $b\in G_k$, $b_i\in G_{m_i}$,
  \begin{align*}
    &\ \gamma (H_kb \cdot H_kc; H_{m_1}b_1 \cdot H_{m_1}a_1, \ldots, H_{m_k}b_k \cdot H_{m_k}a_k) \\
    = &\ H_m\gamma (b\cdot c; b_1\cdot a_1, \ldots, b_k \cdot a_k) \\
    = &\ H_m\gamma ((bc_0,\ldots,bc_l); (b_1 a_{10}, \ldots, b_1 a_{1l}), \ldots, (b_1 a_{k0}, \ldots, b_1 a_{kl})) \\
    = &\ H_m(\gamma (bc_0; b_1 a_{10}, \ldots, b_k a_{k0}), \ldots, \gamma (bc_0; b_1 a_{1l}, \ldots, b_k a_{kl})) \\
  = &\ H_m(\gamma (b; b_1, \ldots, b_k) \gamma (c_0; a_{b^{-1}(1)0}, \ldots, a_{b^{-1}(k)0}), \ldots, \gamma (c_0; a_{b^{-1}(1)l}, \ldots, a_{b^{-1}(k)l})) \\
  = &\ H_m\gamma (b;b_1,\ldots,b_k) \cdot H_m\gamma (c;a_{b^{-1}(1)},\ldots, a_{b^{-1}(k)}) \\
  = &\ \gamma (H_kb; H_{m_1}b_1, \ldots, H_{m_k}b_k) \cdot \gamma (H_kc; H_{m_1}a_1, \ldots, H_{m_k}a_k).
  \end{align*}
  Hence $E\ms{G}/ \ms{H}$ is a simplicial $\ms{G}/ \ms{H}$-operad.
\end{proof}

Any morphism $\psi: \ms{G} \to \ms{G}'$ of group operads induces homomorphisms of simplicial groups $E\psi_n: EG_n \to EG'_n$ which commute with $\gamma$ and are equivariant. $\Ker \psi_n$ is a non-crossed normal subgroup of $G_n$ and acts on $EG_n$. Then $E\ms{G}/\Ker \psi= \{EG_n/ \Ker \psi_n\}_{n\geq 0}$ is a simplicial group and a simplicial $\ms{G}/ \Ker \psi$-operad.

We shall realize any group operad as the fundamental groups operad of its universal operad in the following.

Note that $(EG/H)_0= G/H$, $(EG/H)_1= (G\times G)/H= \{H(a,b) \mid a, b\in G\}$. $H(a,b)$ is a spherical element iff $d_0 H(a,b)= Hb=H$ and $d_1 H(a,b)=Ha=H$, i.e., both $a,b\in H$. Each spherical element $H(a,b)$ has a unique representative $(e,a^{-1}b)$ where $e$ is the identity of $G$. Thus the set of spherical elements is
$$\{H(e,a) \mid a\in H\} \subseteq (EG/H)_1$$
which is actually a group with product $H(e,a) \cdot H(e,b)= H(e,ab)$. Obviously
$$H\to \{H(e,a) \mid a\in H\}, \quad a\mapsto H(e,a)$$
is an isomorphism and natural with respect to homomorphisms $(G,H)\to (G',H')$ of pairs of groups where $H\leq G$, $H'\leq G'$. Recall that $\pi_1(EG/H)$ can be identified with the group of spherical elements $\{H(e,a) \mid a\in H\}$. Hence,

\begin{lem}
  There is a natural isomorphism $H\xra{\cong} \pi_1 (EG/H)$ for $H\leq G$. \qed
\end{lem}

Let $\ms{H}$ be a sub group operad of $\ms{G}$. Note that $\pi_1 (EG_k/H_k)$ can be identified with $\{H_k(e_k,a) \mid a\in H_k\}$. Moreover the fundamental groups operad $\pi_1 (E\ms{G}/ \ms{H})$ can also be identified with the group operad $\{\{H_k(e_k,a) \mid a\in H_k\}\}_{k\geq 0}$ with
$$\gamma (H_k(e_k,a); H_{m_1}(e_{m_1},b_{m_1}), \ldots)= H_m(e_m, \gamma (a; b_{m_1}, \ldots))$$
and $\pi: \{H_k(e_k,a) \mid a\in H_k\}\to S_k$, $\pi H_k(e_k,a)= \pi a$. It is evident that $H_k\cong \{H_k(e_k,a) \mid a\in H_k\}$ and $\ms{H}\cong \{\{H_k(e_k,a) \mid a\in H_k\}\}_{k\geq 0}$.

\begin{thm}
  There is an isomorphism $\ms{H} \cong \pi_1 (E\ms{G}/ \ms{H}) \cong \pi_1 (|E\ms{G}|/ \ms{H})$ of group operads for any sub group operad $\ms{H}\leq \ms{G}$, natural with respect to morphisms of pairs of group operads. In particular $\ms{G}\cong \pi_1 (E\ms{G}/ \ms{G}) \cong \pi_1 (|E\ms{G}|/ \ms{G})$ as group operads. \qed
\end{thm}

\subsection{Characterization and Reconstruction of $K(\pi,1)$ Operads}
\begin{lem}
  Any two universal $\ms{G}$-operads are equivalent.
\end{lem}

This kind of equivalence of universal operads is first observed by P. May \cite{May:1972:GILS} for the case $\ms{G}= \ms{S}$. The cases $\ms{G}= \ms{B}$, $\ms{R}$ are considered in \cite{Fiedorowicz:preprint:SBC}, \cite{Wahl:2001:PhD}, respectively. We shall prove this lemma following P. May's clever idea (cf. \cite{May:1972:GILS}, pages 24--26).

\begin{proof}
  Let $\ms{C}$ and $\ms{C}'$ be two universal $\ms{G}$-operads. Since $\ms{C}$ and $\ms{C}'$ both are contractible, $\ms{C}\times \ms{C}'$ is also contractible. Then $\ms{C}\times \ms{C}'$ is also a universal $\ms{G}$-operad and the two projections $\ms{C} \leftarrow \ms{C} \times \ms{C}' \to \ms{C}'$ are morphisms of $\ms{G}$-operads. So $\ms{C}$ and $\ms{C}'$ are equivalent by definition.
\end{proof}

An algebraic classification of $K(\pi,1)$ topological $\ms{G}$-operads is given as follows.

\begin{thm}
  Let $\ms{C}$, $\ms{C}'$ be two path-connected, locally path-connected and semilocally simply-connected $K(\pi,1)$ covering $\ms{G}$-operads with good basepoints. Then $\ms{C}$, $\ms{C}'$ are equivalent iff $\pi_1 (\ms{C}/ \ms{G}) \cong \pi_1 (\ms{C}'/ \ms{G})$.
\end{thm}

\begin{proof}
  Suppose $\pi_1 (\ms{C}/ \ms{G}) \cong \pi_1 (\ms{C}'/ \ms{G})$. Then the two universal covers $\wt{\ms{C}}$ and $\wt{\ms{C}}'$ of $\ms{C}$ and $\ms{C}'$, respectively, are universal $\pi_1 (\ms{C}/ \ms{G})$-operads, thus equivalent. From the following commutative diagram
  $$\xymatrix{
    & \wt{\ms{C}} \ar[d] & \wt{\ms{C}} \times \wt{\ms{C}}' \ar[d] \ar[l] \ar[r] & \wt{\ms{C}}' \ar[d] \\
    \ms{C} & \wt{\ms{C}}/ \pi_1 \ms{C} \ar[l]_-{\cong} & (\wt{\ms{C}} \times \wt{\ms{C}}')/ \pi_1 \ms{C} \ar[l] \ar[r] & \wt{\ms{C}}'/ \pi_1 \ms{C} \ar[r]^-{\cong} & \ms{C}'   }$$
  $\ms{C}$ and $\ms{C}'$ are equivalent.
\end{proof}

Consequently, a $K(\pi,1)$ operad can be reconstructed from its fundamental groups operad.

\begin{thm}
  Let $\ms{C}$ be a path-connected, locally path-connected and semilocally simply-connected covering $\ms{G}$-operad with a good basepoint. If $\ms{C}$ is $K(\pi,1)$, then $\ms{C}\sim |E\pi_1 (\ms{C}/\ms{G})|/ \pi_1 \ms{C}$.
\end{thm}

\begin{proof}
  $\pi_1 \ms{C}$ is a non-crossed normal sub group operad of $\pi_1 (\ms{C}/\ms{G})$ from the short exact sequence
  $$1\to \pi_1 \ms{C}\to \pi_1 (\ms{C}/ \ms{G})\to \ms{G} \to 1.$$
  Thus $|E\pi_1 (\ms{C}/\ms{G})|/ \pi_1 \ms{C}$ is a $\ms{G}$-operad since $\pi_1 (\ms{C}/\ms{G})/ \pi_1 \ms{C}\cong \ms{G}$. The equivalence then follows from the algebraic characterization.
\end{proof}

\begin{example}
  The little $\infty$-cubes operad $\ms{C}_{\infty}$ is equivalent to $|E\ms{S}|$ \cite{Bar-Ecc:1974:I, May:1972:GILS}; the little 2-cubes operad $\ms{C}_2$ is equivalent to $|E\ms{B}|/ \ms{P}$ \cite{Fiedorowicz:preprint:SBC}. The framed 2-discs operad is equivalent to $|E\ms{R}|/ \Ker \pi$ where $\pi: \ms{R}\to \ms{S}$ is the canonical projection of the ribbon braid groups operad $\ms{R}$ onto $\ms{S}$ \cite{Wahl:2001:PhD}.
\end{example}

\begin{problem}
  Is there analogous characterization and reconstruction of general covering $\ms{G}$-operads? By taking universal cover, it is equivalent to consider only simply connected $\ms{G}$-operads. The case of the little $n$-cubes operads $\ms{C}_n$ for $n\geq 3$ would be of particular interest. Are two path-connected covering $\ms{G}$-operads with good basepoints equivalent if all their homotopy groups operads are isomorphic? Only all $\pi_l$ may not be enough generally since $\pi_l$ are not connected for distinct $l$. So relative homotopy groups or similar objects might be necessary.
\end{problem}

\section{Applications of Group Operads to Homotopy Theory}
For each group operad, there is an associated monad, which is a functor from the category of pointed compactly generated Hausdorff spaces to itself. Via this associated monad, some applications of group operads to homotopy theory are investigated in this section and the main result is a free group model for the canonical stabilization $\lsii X\into \lsoo X$ and particularly a free group model for the homotopy fibre of this stabilization.

\subsection{The Associated Monad of an Operad}
Any topological $\ms{G}$-operad $\ms{C}$ determines a monad from the category of pointed compactly generated Hausdorff spaces to itself. Let $X$ be a pointed space and $G_k$ acts on $X^k$ from the right via $\pi: G_k \to S_k$. Let
$$\ms{C}X = \bigsqcup_{k\geq 0} \ms{C}(k) \times_{G_k} X^k \Big/ \sim$$
with the weak topology, where the equivalence relation is generated by
$$[d_i c,x]\sim [c, d^ix], \textrm{ for } c\in \ms{C}(k), 1\leq i\leq k, x= (x_1,\ldots, x_{k-1}) \in X^{k-1}.$$
This construction $\ms{C}X$ is based on the $\Delta$-set structure on $\ms{C}$ and is similar to the geometric realization of a $\Delta$-set. The two natural maps $\ms{C} (\ms{C}X)\to \ms{C}X$ and $X\to \ms{C}X$ are defined as Construction 2.4 of \cite{May:1972:GILS}.

When studying the associated monads of $\ms{G}$-operads, it is enough to consider only nonsymmetric and symmetric operads. This is because if $\ms{G}$ is non-crossed, $\ms{C}(k) \times_{G_k} X^k= \ms{C}(k)/G_k \times X^k$, thus $\ms{C}X = (\ms{C}/\ms{G})X$; if $\ms{G}$ is crossed with $\pi: \ms{G}\onto \ms{S}$, $\ms{C}(k) \times_{G_k} X^k= \ms{C}(k)/ \Ker \pi \times_{S_k} X^k$, thus $\ms{C}X = (\ms{C}/ \Ker \pi)X$. Refer to \cite{May:1972:GILS} for a detailed theory of the associated monads of operads.

\begin{example}
  Any group operad $\ms{G}= \{G_k\}_{k\geq 0}$ is itself a discrete $\ms{G}$-operad. Thus we have the construction
  $$\ms{G}X= \bigsqcup_{k\geq 0} G_k \times_{G_k} X^k \Big/ \sim.$$
  We claim that this is just the James construction $\ms{J}X$ since $G_k \times_{G_k} X^k= X^k$. As such, we reserve the notation $\ms{G}X$ for another functor defined later.
\end{example}

The following proposition can be verified as the case $\ms{G}= \ms{S}$ (cf. \cite{May:1972:GILS}, Section 2).

\begin{prop}
  For a $\ms{G}$-operad $\ms{C}$, $\ms{C}X$ is a $\ms{C}$-space. \qed
\end{prop}

\begin{thm}[May \cite{May:1972:GILS}, Theorem 2.7]
  $\ms{C}_nX\simeq \lsn X$ for connected pointed CW-complexes.
\end{thm}

A morphism of topological $\ms{G}$-operads $\psi: \ms{C}\to \ms{C}'$ induces a natural map $\psi: \ms{C}X \to \ms{C}'X$ defined by $\psi[c,x]= [\psi(c),x]$. If $\psi$ is an equivalence, there is the following very useful result \cite{May:1972:GILS}.

\begin{lem}[May \cite{May:1972:GILS}, Proposition 3.4]
  If $\psi: \ms{C}\to \ms{C}'$ is an equivalence of path-connected topological $\ms{G}$-operads, then $\psi: \ms{C}X\to \ms{C}'X$ is a homotopy equivalence for connected pointed CW-complexes $X$. \qed
\end{lem}

Proposition 3.4 of \cite{May:1972:GILS} has also an analogous result for non path-connected topological operads with a condition on the structure of the path-connected components.

\subsection{The Associated Monad of a Group Operad}
\begin{defn}
  Let $\ms{G}= \{G_n\}_{n\geq 0}$ be a crossed $\Delta$-group and $X$ a pointed space. define
  $$\ms{G}X= \bigsqcup_{k\geq 0} |EG_k| \times_{G_k} X^k \Big/ \sim$$
  where the equivalence relation is generated by
  $$[a_0, \ldots, a_m; t; d^ix] \sim [d_ia_0, \ldots, d_ia_m; t; x]$$
  for $[a_0, \ldots, a_m;t] \in |EG_n|$, $1\leq i\leq n$ and $x\in X^{n-1}$. Let $\ms{G}_nX$ be the image of $\bigsqcup_{k\geq 0}^n |EG_k| \times_{G_k} X^k$ in $\ms{G}X$.
\end{defn}

If $\ms{G}= \ms{J}$, the above construction is the James construction as well, so it is fine to use the notation $\ms{J}X$. If $\ms{G}$ is a group operad, $\ms{G}X$ is the same as the associated monad of the $\ms{G}$-operad $|E\ms{G}|$.

If $\ms{G}$ is a group operad, these functions $G_n\times G_m\xra{\cong} e_2 \times G_n\times G_m \into G_2\times G_n\times G_m\xra{\gamma} G_{n+m}$ induce a product on $\ms{G}X$.

\begin{prop}
  If $\ms{G}$ is a group operad, then $\ms{G}X$ is a topological monoid; if $\ms{G}$ is moreover crossed, then $\ms{G}X$ is a homotopy abelian monoid. \qed
\end{prop}

\begin{problem}
  What is the homotopy type of $\ms{G} X$ in particular for a group operad $\ms{G}$? For group operads, this is the same to ask about the homotopy type of $\ms{C}X$ for well pointed $K(\pi,1)$ covering $\ms{G}$-operads $\ms{C}$.
\end{problem}

The homotopy type of $\ms{G}X$ is generally unknown yet, but known for a few canonical group operads.

\begin{thm}[James (1955) \cite{James:1955:RPS}]
  $\ms{J}X \simeq \lsi X$ for connected pointed CW-complexes $X$. \qed
\end{thm}

\begin{thm}[Barratt and Eccles (1974) \cite{Bar-Ecc:1974:I}]
  $\ms{S}X \simeq \lsoo X$ for connected pointed CW-complexes $X$. \qed
\end{thm}

Recall that $\ms{C}_2\sim |E\ms{B}|/ \ms{P}$. Thus

\begin{thm}[Fiedorowicz \cite{Fiedorowicz:preprint:SBC}]
  $\ms{B}X \simeq \lsii X$ for connected pointed CW-complexes $X$. \qed
\end{thm}

We next discuss the cases $\ms{G}= \ms{R}$ the ribbon braid groups operad and $\ms{G}= \wt{\ms{S}}$ the sequence of hyperoctahedral groups.

\begin{lem}
  Given a split short exact sequence of groups
  $$1\to N\to G\to H\to 1,$$
  if $X$ is a simplicial $G$-set, then
  $$EG\times_G X\cong EH\times_H (EN\times_N X);$$
  if $X$ is a pointed $G$-space, then
  $$|EG|\times_G X\cong |EH|\times_H (|EN|\times_N X).$$
  Here the product is the one in the category of compactly generated Hausdorff spaces.
\end{lem}

\begin{proof}
  Since the short exact sequence splits, $G\cong N\rtimes H$, i.e., $G\cong N\times H$ as sets and with multiplication
  $$(a_1, h_1)* (a_2, h_2)= (a_1 (h_1a_2h_1^{-1}), h_1h_2), \quad (a_1,h_1), (a_2,h_2)\in N\times H.$$
  Thus $EG\cong E(N\times H)\cong EN\times EH$ as simplicial sets, but $N\rtimes H$ acts on $EN\times EH$ by
  $$((a,h)*(\bar{a}, \bar{h}))= (ah\bar{a}h^{-1}, h\bar{h}), \quad (a,h)\in N\rtimes H, (\bar{a}, \bar{h})\in EN\times EH.$$
  If $X$ is a simplicial $G$-set, then
  $$EG\times_G X\cong (EN\times EH) \times_{N\rtimes H} X= (EN\times EH\times X)/ \stackrel{N\rtimes H}{\sim},$$
  $$(\bar{a}, \bar{h}, x) \stackrel{N\rtimes H}{\sim} ((a,h)*(\bar{a}, \bar{h}), x \cdot (a,h)^{-1})= (ah\bar{a}h^{-1}, h\bar{h}, xh^{-1}a^{-1}),$$
  noting that $(a,h)= (a,1)* (1,h)$, $(a,h)^{-1}= (1,h)^{-1} (a,1)^{-1}$. On the other hand,
  $$EH\times_H (EN\times_N X)= EH \times (EN\times X/ \stackrel{N}{\sim})/ \stackrel{H}{\sim}$$
  where $(\bar{a},x) \stackrel{N}{\sim} (a\bar{a}, xa^{-1})$ and$(\bar{h}, [\bar{a},x]) \stackrel{H}{\sim} (h\bar{h}, [h\bar{a}h^{-1}, xh^{-1}])$. Thus $EH\times_H (EN\times_N X)= (EN\times EH\times X)/ \sim'$ where
  $$(\bar{a}, \bar{h}, x)\sim' (ha\bar{a}h^{-1}, h\bar{h}, xa^{-1} h^{-1}).$$
  Since $ah=ha'$ for some $a'\in N$, $(ah\bar{a}h^{-1}, h\bar{h}, xh^{-1}a^{-1})= (ha'\bar{a}h^{-1}, h\bar{h}, xa'^{-1} h^{-1})$. So the two equivalence relations $\stackrel{N\rtimes H}{\sim}$ and $\sim'$ are actually the same. Hence $EG\times_G X= EH\times_H (EN\times_N X)$. The second part can be proved similarly.
\end{proof}

Given a group $H$, $\{H^k\}_{k\geq 0}$ is a $\Delta$-group (and a simplicial group as well).

\begin{lem}
  Given a split short exact sequence of crossed $\Delta$-groups
  $$1\to \{H^k\}_{k\geq 0} \to \ms{G} \to \ms{G}'\to 1,$$
  if $X$ is a pointed space with an action of $H$, then
  $$\ms{G}X\cong \ms{G}' (|EH|_+ \wedge_H X).$$
\end{lem}

\begin{proof}
  Since the short exact sequence splits, $\ms{G}\cong H\wr \ms{G}'$, thus $\ms{G}X\cong (H\wr \ms{G}') X$. So
  \begin{align*}
    \ms{G}X &\cong \bigsqcup_{k\geq 0} |E(H^k\rtimes G'_k)| \times_{H^k\rtimes G'_k} X^k \Big/ \sim \cong \bigsqcup_{k\geq 0} |EG'_k| \times_{G'_k} (|EH|\times_H X)^k \Big/ \sim \\
    &\cong \bigsqcup_{k\geq 0} |EG'_k| \times_{G'_k} (|EH|_+ \wedge_H X)^k \Big/ \sim = \ms{G}' (|EH|_+ \wedge_H X)
  \end{align*}
  by the previous lemma.
\end{proof}

Recall that for the ribbon braid groups operad $\ms{R}$, there is the following split short exact sequence of crossed $\Delta$-groups
$$1\to \{\Z^k\}_{k\geq 0}\to \ms{R}\to \ms{B}\to 1.$$

\begin{cor}
  $\ms{R}X \simeq \lsii (|E\Z|_+ \wedge_{\Z} X)$ for connected pointed CW-complexes $X$ with a $\Z$-action. \qed
\end{cor}

Note that $R_1=\Z$ acts on each $R_k$ by $\gamma: R_1\times R_k\to R_k$. This action does not contribute to the the construction $\ms{R}X$, but to the structure on $\ms{R}X$.

If the $\Z$-action on $X$ is trivial, then $\ms{R}X \simeq \lsii (|B\Z|_+ \wedge X) \simeq \lsii (S^1_+ \wedge X)$ which is given in Wahl's Ph.D. thesis \cite{Wahl:2001:PhD}.

Also recall that for the sequence of hyperoctahedral groups $\wt{\ms{S}}= \{\wt{S}_k\}_{k\geq 0}$ there is the following split short exact sequence of crossed $\Delta$-groups
$$1\to \{(\Z/2)^k\}_{k\geq 0}\to \wt{\ms{S}}\to \ms{S}\to 1.$$

\begin{cor}
  $\wt{\ms{S}}X \simeq \lsoo (|E(\Z/2)|_+ \wedge_{\Z/2} X)$ for connected pointed CW-complexes $X$ with a $\Z/2$-action. \qed
\end{cor}

A detailed theory of $\ms{S}X$ is developed in \cite{Bar-Ecc:1974:I, Bar-Ecc:1974:II, Bar-Ecc:1974:III}. Many parts of it actually are also valid for general group operads. Influenced by this special case, we shall investigate some more aspects of $\ms{G}X$ and their applications in the rest of this section.

\subsection{Freeness and Group Completion of $\ms{G}X$}
\begin{prop}
  For a group operad $\ms{G}$, $\ms{G}X$ is a free monoid.
\end{prop}

We give a proof for the simplicial version following the proof for the case $\ms{G}= \ms{S}$ (\cite{Bar-Ecc:1974:I}, Proposition 3.11). The proof for the topological version is similar but we need to deal with the geometric realization of simplicial sets. For a simplicial set $X$, let
$$E\ms{G}X= \bigsqcup_{n\geq 0} EG_n\times X^n \Big/ \sim$$
where the equivalence relation is the same as the one for the topological version. Explicitly, $\sim$ is generated by
$$(\bar{a};x) \sim (c\bar{a}; xc^{-1}), \quad (\bar{a}; d^ix) \sim (d_i\bar{a};x)$$
for $\bar{a}= (a_0, \ldots, a_l)\in (EG_n)_l$, $c\in G_n$, $x\in X^n$, where $d_i\bar{a}= (d_ia_0, \ldots, d_ia_l)$. Let
$$\times: G_n\times G_m \xra{\cong} e_2\times G_n\times G_m \into G_2\times G_n\times G_m \xra{\gamma} G_{n+m}, \quad (a,b)\mapsto a\times b.$$
This product induces a product
$$\times: EG_n\times EG_m\to EG_{n+m}, \quad \bar{a}\times \bar{b}= (a_0, \ldots, a_l)\times (b_0, \ldots, b_l) \mapsto (a_0\times b_0, \ldots, a_l\times b_l)$$
and thus a product
$$E\ms{G}X \times E\ms{G}X \to E\ms{G}X, \quad [\bar{a};x] \times [\bar{b};y] \mapsto [\bar{a} \times \bar{b}; x,y].$$
$[\bar{a};x]\in E\ms{G}X$ is \textbf{reducible}, if $[\bar{a};x]= [\bar{b};y] \times [\bar{c};z]$ for some $[\bar{b};y]$, $[\bar{c};z]$ none of which is the identity element. If no coordinate of $x\in X^n$ is the basepoint, then $[\bar{a};x]$ is \textbf{irreducible} iff there does not exist $b\in G_n$ such that $b\bar{a}\in EG_r\times EG_{n-r}$ for some $r$, $1\leq r\leq n-1$. To show that $E\ms{G}X$ is free, it is sufficient to show that each element of $\ms{G}X$ can be written in a unique way as a product of irreducible elements. This is equivalent to the following lemma.

\begin{lem}
  If $[\bar{a};x]$ and $[\bar{a}';x']$ are irreducible and $[\bar{a};x] \times [\bar{b};y]= [\bar{a}';x'] \times [\bar{b}';y']$, then $[\bar{a};x]= [\bar{a}';x']$, $[\bar{b};y]= [\bar{b}';y']$.
\end{lem}

Thus it remains to prove this lemma.

\begin{proof}
  Suppose no coordinates of $x,x',y,y'$ are the basepoint, and $\bar{a}= (\ldots, a_i, \ldots) \in EG_m$, $\bar{a}'= (\ldots, a'_i, \ldots) \in EG_n$, $\bar{b}= (\ldots, b_i, \ldots) \in EG_p$, $\bar{b}'= (\ldots, b'_i, \ldots) \in EG_q$. Then $m+p=n+q$. $[\bar{a};x] \times [\bar{b};y]= [\bar{a}';x'] \times [\bar{b}';y']$ implies that there is $c\in G_{m+p}$ such that $(\bar{a} \times \bar{b}; x,y)= (c(\bar{a}' \times \bar{b}'); (x',y') c^{-1})$. We need to show $(\bar{a};x)= (c_1\bar{a}'; x'c_1^{-1})$ and $(\bar{b};y)= (c_2 \bar{b}; y' c_2^{-1})$ for some $c_1, c_2$.

  Suppose $m=n$. Since $a_i\times b_i= c(a'_i\times b'_i)\in G_{m+p}= G_{n+q}$, we have $c= (a_i\times b_i) (a'_i\times b'_i)^{-1}\in G_m \times G_p$. Thus $c=c_1\times c_2$ for some $c_1\in G_m$ and $c_2\in G_p$ which are the required ones.

  Next we show that we must have $m=n$ by contradiction. Without loss of generality we may suppose $m<n$. If there is some $k\in [m]$ such that $c(k)>n$, then $(a_i\times b_i)(k)= a_i(k) \in [m]$ but $(c(a'_i\times b'_i))(k)= (a'_i\times b'_i)(c(k))\in \{n+1, \ldots, n+q\}$, contradicting that $a_i\times b_i= c(a'_i\times b'_i)$. Thus $c[m]\subset [n]$ and $c^{-1} \{n+1, \ldots, n+q\} \subset \{m+1, \ldots, m+p\}$. Then
  \begin{align*}
    d_{c^{-1} \{n+1, \ldots, n+q\}} (c(\bar{a}' \times \bar{b}')) &= (d_{c^{-1} \{n+1, \ldots, n+q\}}c) d_{\{n+1, \ldots, n+q\}} (\bar{a}' \times \bar{b}') \\
    &= (d_{c^{-1} \{n+1, \ldots, n+q\}}c)\bar{a}',
  \end{align*}
  but on the other hand,
  \begin{align*}
    d_{c^{-1} \{n+1, \ldots, n+q\}} (c(\bar{a}' \times \bar{b}')) &= d_{c^{-1} \{n+1, \ldots, n+q\}} (\bar{a} \times \bar{b}) \\
    &= \bar{a} \times d_{\{c^{-1}(n+1)-m, \ldots, c^{-1}(n+q)-m\}} \bar{b},
  \end{align*}
  contradicting the irreducibility of $[\bar{a}',x']$.
\end{proof}

For a monoid $M$, recall that its universal group $UM$, cf. Proposition 4.2 of \cite{Bar-Ecc:1974:I}, is constructed as follows. Let $FM$ be the free group on the pointed (by the unit) set $M$ and denote the image of $a\in M$ in $FM$ by $[a]$. Let $N$ be the normal subgroup of $FM$ generated by all elements of the form $[a] \cdot [b] \cdot [ab]^{-1}$, $a,b\in M$. Then define $UM= FM/N$. $UM$ has the property that $UM$ is a free group if $M$ is a free monoid. Consequently $U\ms{G}X$ is a free group since $\ms{G}X$ is a free monoid.

\begin{prop}
  The inclusion $\ms{G}X\to U\ms{G}X$ is a homotopy equivalence for connected pointed CW-complexes.
\end{prop}

\begin{proof}
  The proof for the case $\ms{G}= \ms{S}$ (\cite{Bar-Ecc:1974:I}, Corollary 5.4) applies to this general situation as well.
\end{proof}

A morphism $\psi: \ms{G}\to \ms{G}'$ of group operads naturally induces maps $\psi: \ms{G}X\to \ms{G}'X$ and $U\psi: U\ms{G}X\to U\ms{G}'X$.

\begin{prop}
  For a morphism $\psi: \ms{G}\to \ms{G}'$ of group operads, the following diagram is commutative
  $$\xymatrix{
    U\ms{G}X \ar[r]^-{U\psi} & U\ms{G}'X \\
    \ms{G}X \ar@{^(->}[u] \ar[r]^-{\psi} & \ms{G}'X \ar@{^(->}[u]   }$$
    If $\ms{G}$ is non-crossed, the obvious morphism $\ms{J}\into \ms{G}\onto \ms{J}$ is the identity and so are $\ms{J}X \into \ms{G}X \onto \ms{J}X$ and $U\ms{J}X \into U\ms{G}X \onto U\ms{J}X$, thus $\ms{J}X$ is a retract of $\ms{G}X$ and $U\ms{J}X$ is a retract of $U\ms{G}X$. \qed
\end{prop}

\subsection{Some Applications to $\lsii X$}
Combining $\ms{B}X$, $\ms{S}X$, $U\ms{B}X$ and $U\ms{S}X$ together, we can have a free group model for the canonical inclusion $\lsii X\into \lsoo X$.

\begin{prop}
  For a connected pointed CW-complex $X$, there is the following homotopy commutative diagram
  $$\xymatrix@=1.2pc{
    \ms{B}X \ar@{->>}[rr]^-{\pi} \ar@{<->}[dd]_{\simeq}
      &  & \ms{S}X \ar@{<->}[dd]^{\simeq}        \\
    & \ms{J}X \ar@{_(->}[ul] \ar@{^(->}[ur] \ar[dd]^(.3){\simeq} \\
    \lsii X \ar@{^(->}'[r][rr]
      &  & \lsoo X                \\
    & \lsi X \ar@{_(->}[ul] \ar@{^(->}[ur]        }$$
    with each vertical arrow a homotopy equivalence.
\end{prop}

\begin{proof}
  Let $X$ be a connected pointed CW-complex. By definition, $\ms{B}X= |E\ms{B}|(X)$ and $\ms{S}X= |E\ms{S}|(X)$. Commutativity of the top and bottom triangles is obvious. Note that
  $$\xymatrix{
    \ms{C}_1X \ar[d]_{\simeq} \ar@{^(->}[r] & \ms{C}_2X \ar[d]^{\simeq} \ar@{^(->}[r] & \cdots \ar@{^(->}[r] & \ms{C}_{\infty} X \ar[d]^{\simeq} \\
    \lsi X \ar@{^(->}[r] & \lsii X \ar@{^(->}[r] & \cdots \ar@{^(->}[r] & \lsoo X   }$$
  is naturally commutative. Let $\wt{\ms{C}}_2$ be the universal cover of $\ms{C}_2$. Obviously the universal cover of $\ms{C}_{\infty}$ is $\wt{\ms{C}}_{\infty}= \ms{C}_{\infty}$. Note that the two projections $|E\ms{B}| \leftarrow |E\ms{B}| \times \wt{\ms{C}}_2 \to \wt{\ms{C}}_2$ are equivalences of $\ms{B}$-operads and the two projections $|E\ms{S}| \leftarrow |E\ms{S}| \times \wt{\ms{C}}_{\infty} \to \wt{\ms{C}}_{\infty}$ are equivalences of $\ms{S}$-operads. We have the following commutative diagram of $\ms{B}$-operads and $\ms{S}$-operads
  $$\xymatrix{
    |E\ms{B}| \ar@{->>}[d] & |E\ms{B}| \times \wt{\ms{C}}_2  \ar[l]_-{\sim} \ar[d] \ar[r]^-{\sim} & \wt{\ms{C}}_2 \ar@{^(->}[d] \ar[r] & \ms{C}_2 \ar@{^(->}[d] \\
    |E\ms{S}| & |E\ms{S}| \times \wt{\ms{C}}_{\infty} \ar[l]_-{\sim} \ar[r]^-{\sim} & \wt{\ms{C}}_{\infty} \ar@{=}[r] & \ms{C}_{\infty}   }$$
  thus there is the following commutative diagram
  $$\xymatrix{
    |E\ms{B}|(X) \ar@{->>}[d] & (|E\ms{B}| \times \wt{\ms{C}}_2) (X) \ar[l]_-{\simeq} \ar[d] \ar[r]^-{\simeq} & \wt{\ms{C}}_2 X \ar@{^(->}[d] \ar[r] & \ms{C}_2 X \ar@{^(->}[d] \ar[r]^-{\simeq} & \lsii X \ar@{^(->}[d] \\
    |E\ms{S}|(X) & (|E\ms{S}| \times \wt{\ms{C}}_{\infty}) (X) \ar[l]_-{\simeq} \ar[r]^-{\simeq} & \wt{\ms{C}}_{\infty}X \ar@{=}[r] & \ms{C}_{\infty}X \ar[r]^-{\simeq} & \lsoo X  }$$
  Hence the back square is commutative. Commutativity of the left square comes from the following commutative diagram
  $$\xymatrix{
    \ms{J}X \ar@{^(->}[d] \ar@{^(->}[dr] \ar@{^(->}[rrr]^-{\simeq} &&& \ms{C}_1X \ar@{^(->}[d] \ar[r]^-{\simeq} & \lsi X \ar@{^(->}[d] \\
    |E\ms{B}|(X) & (|E\ms{B}| \times \wt{\ms{C}}_2) (X) \ar[l]_-{\simeq} \ar[r]^-{\simeq} & \wt{\ms{C}}_2 (X) \ar@{=}[r] & \ms{C}_2X \ar[r]^-{\simeq} & \lsii X   }$$
  where the inclusion $\ms{J}X \into (|E\ms{B}| \times \wt{\ms{C}}_2) (X)$ is induced by the two canonical inclusions $\ms{J}X \into |E\ms{B}|(X)$ and $\ms{J}X \into \wt{\ms{C}}_2X$. Proof for the commutativity of the right square is similar.
\end{proof}

\begin{thm}
  For a connected pointed CW-complex $X$, there is the following homotopy commutative diagram
  $$\xymatrix{
    U\ms{B}X \ar@{->>}[r]^-{U\pi} & U\ms{S}X \\
    \ms{B}X \ar@{^(->}[u]^{\simeq} \ar@{->>}[r]^-{\pi} & \ms{S}X \ar@{^(->}[u]_{\simeq} \\
    \lsii X \ar@{<->}[u]^{\simeq} \ar@{^(->}[r] & \lsoo X \ar@{<->}[u]_{\simeq}   }$$
  Thus the free group bundle $U\pi: U\ms{B}X \onto U\ms{S}X$ is a model of the canonical inclusion $\lsii X \into \lsoo X$, In particular, the free group $(U\pi)^{-1}(*)$ is a model of its homotopy fibre. \qed
\end{thm}

\begin{problem}
  Use the free group model $U\pi: U\ms{B}X \to U\ms{S}X$ and its fibre to study $\lsii X$ and its relationship with $\lsoo X$.
\end{problem}

It is common that a good filtration of a space may provide considerable information of the space. So finding good filtrations of $\lsii X$ may be helpful to understand $\lsii X$. For a $\ms{G}$-operad $\ms{C}$, $\ms{C}X$ has an obvious truncated filtration $\{F_n\ms{C}X\}_{n\geq 0}$ where $F_n\ms{C}X$ is the image of $\bigsqcup_{k=0}^n \ms{C}(k) \times_{G_k} X^k$. This filtration has been well studied, as well as the filtration $\om \ms{J}_n \si X$ for $\lsii X$, where $\ms{J}_nY$ is the truncated filtration of the James construction $\ms{J}Y$.

Using the two models $\ms{B}X$ and $\ms{S}X$, a few more canonical filtrations of $\lsii X$ can be constructed. Generally, for a $\ms{G}$-operad $\ms{C}= \{\ms{C}(k)\}_{k\geq 0}$, each $\ms{C}(k)$ canonically generates a $\ms{G}$-suboperad and thus they give a filtration of $\ms{C}$ which may be called the operadic filtration of $\ms{C}$ (as each term is a suboperad). The first stage of the operadic filtration of $\ms{B}X$ and $\ms{S}X$ is exactly $\ms{J}X$. By taking the preimage of the operadic filtration of $\ms{S}X$ under the map $\ms{B}X \to \ms{S}X$, one gets another filtration of $\ms{B}X$ with the first two stages $\ms{J}X \into \ms{P}X$. Moreover, note that Smith's models \cite{Smith:1989:SGM} for $\lsn X$ for all $n$, constitute a filtration of $\ms{S}X$ and thus their preimages constitute another filtration of $\lsii X$.

Finally we remark that three canonical self-maps of $\lsii X$ can be represented by three types of group homomorphisms between braid groups, via the model $\ms{B}X$ and $\ms{B}X \leftarrow (|E\ms{B}| \times \wt{\ms{C}}_2)(X) \to \ms{C}_2X$. F. Cohen \cite{Cohen:1976:SCSSM} constructs a self-map of $\lsn X$ from a sequence of self-maps of little $n$-cubes and deduces a splitting of $\si \lsn X$. When $n=2$, this sequence of self-maps of little $2$-cubes actually can be represented by the sequence of reflection homomorphisms $\chi_k: B_k\to B_k$ sending the standard generators to their inverse. Moreover, the degree $k$ map $\om^2 [k]: \lsii X\to \lsii X$ and the power $k$ map $k: \lsii X\to \lsii X$ can also be represented, respectively, by the following two types of homomorphism
$$\gamma (-;e_k^{(n)}): B_n\to B_{kn}, \quad \gamma (e_k; (-)^{(k)}): B_n\to B_{kn}.$$
In addition, the splitting theorem of the James construction can be generalized to non-crossed operads, and the stable splitting of $\ms{S}X$ \cite{Bar-Ecc:1974:III} can be generalized to crossed group operads.

\section*{Acknowledgements}
This paper is the first part of the author's Ph.D. thesis submitted recently to National University of Singapore (NUS). The author would like to express his sincerest gratitude to his supervisor, Professor Jie Wu, for his helpful advices, support and encouragement during the last five years. The author is grateful to Professors Jon Berrick, Fred Cohen, Muriel Livernet, Stephen Theriault and Assistant Professor Fei Han for their kind and/or helpful discussion, comments, support and/or help during recent years. Thanks should also go to NUS for providing the author a chance and scholarship to pursue a Ph.D. degree, and to the Department of Mathematics for providing a comfortable environment for study, opportunities for training his teaching ability, and financial support for his fifth year. It is the author's pleasure to thank the three examiners of his thesis for their helpful comments, suggestions and numerous minor corrections of typos, etc. In particular, defining other types of operads as an ``operad with extra structure'' is suggested by one examiner.

\bibliography{MyReference}
\bibliographystyle{amsplain}

\end{document}